\documentclass[11pt]{amsart}
\usepackage[utf8]{inputenc}

\usepackage{xcolor}
\usepackage{amsmath}
\usepackage{amssymb}
\usepackage{mathtools}
\usepackage{mathpazo}
\usepackage{eucal}
\usepackage{amsthm}
\usepackage{amscd}
\usepackage{hyperref}
\usepackage{soul}
\usepackage{url}
\usepackage{verbatim}

\usepackage{amssymb,}
\usepackage[]{amsmath, amsthm, amsfonts,graphicx, amscd,}
\usepackage[all,cmtip]{xy}
\input amssym.def \input amssym

\newtheorem{thm}{Theorem}[section]
\newtheorem{cor}[thm]{Corollary}
\newtheorem{claim}[thm]{Claim}
\newtheorem {fact}[thm]{Fact}
\newtheorem*{thmstar}{Theorem}
\newtheorem{prop}[thm]{Proposition}

\newtheorem {lem}[thm]{Lemma}

\newtheorem{ques}[thm]{Question}

\theoremstyle{remark}
\newtheorem{rem}[thm]{Remark}
\newtheorem{np*}{Non-Proof}

\theoremstyle{definition}
\newtheorem{defn}[thm]{Definition}

\newtheorem{exam}[thm]{Example}

\newcommand{\pd}[2]{\frac{\partial #1}{\partial #2}}

\def\Ind{\setbox0=\hbox{$x$}\kern\wd0\hbox to 0pt{\hss$\mid$\hss} \lower.9\ht0\hbox to 0pt{\hss$\smile$\hss}\kern\wd0}
\def\Notind{\setbox0=\hbox{$x$}\kern\wd0\hbox to 0pt{\mathchardef \nn=12854\hss$\nn$\kern1.4\wd0\hss}\hbox to 0pt{\hss$\mid$\hss}\lower.9\ht0 \hbox to 0pt{\hss$\smile$\hss}\kern\wd0}
\def\ind{\mathop{\mathpalette\Ind{}}}

\def\trdeg{\operatorname{trdeg}}

\newcommand{\m}{\mathbb }
\newcommand{\mc}{\mathcal }

\newcommand{\Q}{\mathbb Q}

\newcommand{\tp}{\operatorname{tp}}

\newcommand{\gen}[1]{\left\langle#1\right\rangle}

\def \Q {{\mathbb Q}}

\def \hat {\widehat}
\def \bar{\overline}

\def \ACF0 {{\rm ACF}_0}

\def \tp {{\rm tp}}

\def\< {\Lngle}
\def \> {\rangle}

\def \minusdot{\hbox{\ {$-$} \kern -.86em\raise .2em \hbox{$\cdot \
$}}}

\def \tilde {\widetilde}

\def \deg { {\rm deg} }
\def\hat {\widehat}

 \def \| {\kern -.3em \restriction \kern -.3em}

 \def \< { \langle}
 \def \> { \rangle}

  \def \Aff {{\rm Aff}}

  \def\alg {{\rm alg}}
  \def \ld{{\rm ldim}}
  \font \small=cmsy10
\def \dnf{ \hbox{\ \ {\small\char '152}\kern -.65em \lower .4em \hbox{$\smile$}}}
\def \forks{\hbox{\ \ {\small \char '152} \kern -.7em \lower .1em
\hbox{\it /} \kern -1.1em \lower .4em \hbox{$\smile$}}}

\begin{document}

\makeatletter

\title{On the equations of Poizat and Li\'enard}

\@namedef{subjclassname@2020}{%
  \textup{2020} Mathematics Subject Classification}
\makeatother

\author[J. Freitag]{James Freitag}
\address{James Freitag, University of Illinois Chicago, Department of Mathematics, Statistics,
and Computer Science, 851 S. Morgan Street, Chicago, IL, USA, 60607-7045.}
\email{jfreitag@uic.edu}

\author[R. Jaoui]{R\'emi Jaoui}
\address{Mathematisches Institut, Albert-Ludwigs-Universit\"at, Ernst-Zermelo-Str. 1, 79104 Freiburg}
\email{remi.jaoui@math.uni-freiburg.de}

\author[D. Marker]{David Marker}
\address{David Marker, University of Illinois Chicago, Department of Mathematics, Statistics,
and Computer Science, 851 S. Morgan Street, Chicago, IL, USA, 60607-7045.}
\email{marker@uic.edu}

\author[J. Nagloo]{Joel Nagloo}
\address{Joel Nagloo, University of Illinois Chicago, Department of Mathematics, Statistics,
and Computer Science, 851 S. Morgan Street, Chicago, IL, USA, 60607-7045.}
\email{jnagloo@uic.edu}

\thanks{J. Freitag is partially supported by NSF CAREER award 1945251 and the Fields Institute for Research in the Mathematical Sciences. R. Jaoui is partially supported by the ANR-DFG program GeoMod
(Project number 2100310201). D. Marker was partially supported by the Fields Institute for Research in the Mathematical Sciences.  J. Nagloo is partially supported by NSF grant DMS-2203508.}

\date{\today}

\maketitle

\begin{abstract} We study the structure of the solution sets in universal differential fields of certain differential equations of order two, the Poizat equations, which are particular cases of Liénard equations. We give a necessary and sufficient condition for strong minimality for equations in this class and a complete classification of the algebraic relations for solutions of strongly minimal Poizat equations. We also give an analysis of the non strongly minimal cases as well as applications concerning the Liouvillian and Pfaffian solutions of some Liénard equations.
\end{abstract} 

\tableofcontents

\section{Introduction}
Our manuscript deals with three prominent topics in algebraic differential equations and their connections to each other, especially interpreted in the context of rational planar vector fields with constant coefficients.

\subsection{Model theory} Strong minimality is an important notion emerging from stability theory, and in the context of differential equations, the notion has a concrete interpretation in terms of functional transcendence. The zero set of a differential equation, $X$, with coefficients in a differential field $K$ is \emph{strongly minimal} if and only if (1) the equation is irreducible over $K^{alg}$ and (2) given any solution $f$ of $X$ and any differential field extension $F$ of $K$, $$\text{trdeg}_F  F\langle f \rangle =  \text{trdeg}_K K \langle f \rangle \text{ or } 0.$$ 
Here $K \langle f \rangle $ denotes the differential field extension of $K$ generated by $f$. 

To the non-model theorist, it likely isn't obvious from the definition, but strong minimality has played a central role in the model theoretic approach to algebraic differential equations. Two factors seem to be important in explaining the centrality of the notion. First, once strong minimality of an equation is established, the trichotomy theorem, a model theoretic classification result, along with other model theoretic results can often be employed in powerful ways \cite{HPnfcp, nagloo2014algebraic}. Second, among nonlinear differential equations, the property seems to hold rather ubiquitously; in fact there are theorems to this effect in various settings \cite{devilbiss2021generic, jaoui2019generic}. Even for equations which are not themselves minimal, there is a well-known decomposition technique, \emph{semi-minimal analysis}\footnote{Definitions of model theroetic notions can be found in section \ref{prelim}.} \cite{moosa2014model}, which often allows for the reduction of questions to the minimal case. 

Establishing the notion has been the key step to resolving a number of longstanding open conjectures \cite{casale2020ax, nagloo2014algebraic}. Despite these factors, there are few enough equations for which the property has been established that a comprehensive list of such equations appears in \cite{devilbiss2021generic}. In this manuscript, we generalize results of Poizat \cite{poizat1977rangs} and Brestovski \cite{brestovski1989algebraic} by showing that

\begin{thmstar}
The set of solutions of
\begin{equation*}
z''={z'}f(z),\;\;\;\;z'\neq 0
\end{equation*}
where $f(z) \in \m C(z)$ is strongly minimal \emph{if and only if} $f(z)$ is not the derivative of some $g(z) \in \m C (z)$. 
\end{thmstar}
In addition to giving a complete characterization for this class of equations, our proof gives a new technique for establishing strong minimality which relies on valuation theoretic arguments about the field of Puiseux series. In the strongly minimal case, we give a precise characterization of the algebraic relations between solutions (and their derivatives) of equations in our class (discussed in the third part of this introduction).

\sloppy When the equation is not strongly minimal, we show that it must be \emph{nonorthogonal to the constants}. The solution set $X$ is orthogonal to the constants if, perhaps over some differential field extension $F $ of $k$, there is a solution $a$ of $X$ such that $F \langle a \rangle $ contains a constant which is not in $F^{alg}$. Again, to non-model theorists, it likely isn't obvious that this condition should play a such a central role as it does.

With respect to the semi-minimal analysis of the generic type $p(z)$ of the equation, three possibilities are a priori possible in this case: 
\begin{enumerate} 
\item $p(z)$ is internal to the constants (this is a strengthening of nonorthogonality to the constants). 
\item $p(z)$ is 2-step analyzable in the constants. 
\item For generic $c \in \m C$, $z' = \int f(z) dz +c$ is orthogonal to the constants, and in the semi-minimal analysis of $p(z)$ there is one type nonorthogonal to the constants and one trivial type. 
\end{enumerate} 

In Section \ref{nonmin}, we show that any of the three possibilities can occur within the non-minimal equations in our family, providing concrete examples of each case. This type of analysis is done in Section \ref{nonmin} and is similar to the results of \cite{jin2020internality} (who did this analysis for a different class of order two equations). Our analysis involves work along the lines of the techniques of \cite{HrIt, notmin}, and there are a number of results of independent interest developed in the course of this analysis. %For instance, in Section \ref{formssection}, we give a proof of a result of Hrushovski and Itai \cite{HrIt} using our language of differential forms. 

\subsection{Special solutions and integrability} One of the fundamental problems of algebraic differential equations is to express the solutions of a differential equation or the first integral of a vector field by some specific \emph{known functions}\footnote{e.g. rational, algebraic, elementary, Liouvillian.} and arbitrary constants or to \emph{show that this is impossible}. In this manuscript, we develop the connection between various such impossibility results for solutions and the notions coming from model theory described above. In particular, we establish results for equations of Li\'enard type: 
\begin{equation*}
    x'' (t) + f(x) x' (t) + g(x) =0,
\end{equation*}
for $f(x),g(x)$ rational functions. Notice that the equations of this type generalize the Brestovski-Poizat type equations described above. This family of equations has its origins in the work of Li\'enard \cite{Lie1, Lie2} and has been the subject of study from a variety of perspectives in large part due to its important applications in numerous scientific areas. See \cite{harko2014class} and the references therein for numerous applications. The class of equations has been intensely studied with respect to finding explicit solutions and integrability, mainly from the point of view of Liouvillian functions. We give a review of the existing results in Section \ref{prevlie}. The connections between these model theoretic notions and the equation having certain special solutions are known to some experts, but there does not seem to be any account of these connections in the literature. Our approach makes use of model theoretic notions and, in particular, a recent specialization theorem of the second author \cite{jaoui2020corps}.

\subsection{Algebraic relations between solutions} Though establishing the strong minimality of a differential equation is itself sometimes a a motivational goal, in many cases it is just the first step in a strategy to classify the algebraic relations between solutions of the equation. See for instance \cite{jaoui2019generic}, where this strategy is employed for generic planar vector fields.  In \cite{casale2020ax}, this strategy is used to prove the Ax-Lindemann-Weierstrass theorem for the automorphic functions associated with Fuchsian groups. Sections \ref{formssection} and \ref{algrelsection} are devoted to classifying the algebraic relations between the strongly minimal equations of Brestovski-Poizat type. 

\begin{thmstar}
Let $f_1(z),\ldots, f_n(z) \in \mathbb{C}(z)$ be rational functions such that each $f_i(z)$ is not the derivative of a rational function in $\mathbb{C}(z)$ and consider for $i = 1,\ldots, n$, $y_i$ 	a solution of 
$$(E_i): y''/y' = f_i(y)$$
Then $trdeg_\mathbb{C}(y_1,y'_1,\ldots, y'_n,y_n) = 2n$ unless for some $i \neq j$ and some $(a,b) \in \mathbb{C}^\ast \times \mathbb{C}$, $y_i = ay_j + b$. In that case,  we also have $f_i(z) = f_j(az + b)$.
\end{thmstar}
Much of the analysis of Section \ref{formssection} is of independent interest. Indeed, in Section \ref{5.1} we set up the formalism of volume forms, vector fields, and Lie derivatives quite generally. In Section \ref{5.2} we give a proof of a result of Hrushovski and Itai \cite{HrIt} using our formalism. In Section \ref{5.3}, we develop and use formalism around the Lie algebra of volume forms to show that for equations in our class, characterizing algebraic relations between solutions and their derivatives follows from characterizing \emph{polynomial relations} between solutions (with no derivatives). Following this, in Section \ref{algrelsection}, we give a precise characterization of the polynomial relations which can appear. In Section \ref{nonmin} we turn towards the nonminimal case and characterize the type of semi-minimal analysis which can appear for the equations from the class and make some remarks regarding the implications of this analysis on the dimension order property (DOP). 

\subsection{Organization of the paper} 
Section \ref{prelim} contains the basic definitions and notions from model theory and the model theory of differential fields that we use throughout the paper. The basic setup of other topics is mostly carried out in the respective sections throughout the paper. In Section \ref{strminsection} we characterize strong minimality for equations of a generalized Brestovski-Poizat form. In Section \ref{integrabilityandnot}, we give a brief introduction to integrability and various special classes of solutions, overview the extensive previous work for equations of Li\'enard type, and prove our results on the existence of Liouvillian solutions to Li\'enard equations. In Sections \ref{formssection} and \ref{algrelsection} with classify the algebraic relations between solutions of strongly minimal equations in the generalized Brestovski-Poizat class. In Section \ref{nonmin} we analyze the nonminimal equations of the class. 

\section{Preliminaries} \label{prelim}

Throughout, $(\mathcal{U},\delta)$ will denote a saturated model of $DCF_0$, the theory of differentially closed fields of characteristic zero with a single derivation. So $\mathcal{U}$ will act as a ``universal'' differential field in the sense of Kolchin. We will also assume that its field of constants is $\m C$. We will be using standard notations: given a differential field $K$, we denote by $K^{alg}$ its algebraic closure and if $y$ is a tuple from $\mathcal{U}$, we use $K\gen{y}$ to denote the differential field generated by $y$ over $K$, i.e. $K\gen{y}=K(y,\delta (y),\delta^2(y),\ldots)$. We will sometimes write $y'$ for $\delta (y)$ and similarly $y^{(n)}$ for $\delta^n (y)$.

Recall that a Kolchin closed subset of $\mathcal{U}^n$ is the vanishing set of a finite system of differential polynomials equations and by a definable set we mean a finite Boolean combination of Kolchin closed sets. In the language $L_{\delta}=(+,-,\times,0,1,\delta)$ of differential rings, these are precisely the sets defined by quantifier free $L_{\delta}$-formulas. Since $DCF_0$ has quantifier elimination, these are exactly {\em all} the definable sets. If a definable set $X$ in $\mathcal{U}^n$ is defined with parameters from a differential field $K$, then we will say that $X$ is defined over $K$. Given such an $X$, we define the {\em order} of $X$ to be $ord(X)=sup\{\text{tr.deg.}_FF\langle y \rangle:y\in X\}$ where $F$ is any differential field over which $X$ is defined. We call an element $y\in X$ {\em generic over $K$} if $\text{tr.deg.}_KK\langle y\rangle = ord(X)$.

As mentioned in the introduction, strong minimality is the first central notion that is studied in this paper:
\begin{defn}
A definable set $X$ is said to be {\em strongly minimal} if it is infinite and for every definable subset $Y$ of $X$, either $Y$ or $X\setminus Y$ is finite.
\end{defn}
It is not hard to see that  $\mathbb{C}$, the field of constants, is strongly minimal.
\begin{rem}\label{SMLienard}We will be mainly concerned with equations of Li\'enard type and in that case, we have a nice algebraic characterization of strong minimality: 
Let $\mathcal{C}\subset\mathbb{C}$ be a finitely generated subfield. Let $X$ be defined by an ODE of the form $y^{(n)}=f(y,y',\ldots,y^{(n-1)})$, where $f$ is rational over $\mathcal{C}$. Then $X$ (or the equation) is strongly minimal if and only if for any differential field extension $K$ of $\mathcal{C}$ and solution $y\in X$, we have that $\text{tr.deg.}_KK\gen{y}=0$ or $n$. 

If $X$ is given as a vector field on the affine plane, then if $X$ is strongly minimal there are no invariant algebraic curves of the vector field (if there were, the generic solution of the system of equations given by $X$ and the curve would violate the transcendence condition we describe in the previous paragraph). For instance, the equation $z''=z \cdot z' $ studied by Poizat \cite{poizat1977rangs} is not strongly minimal, but the definable set $z''=z \cdot z' , \, z'\neq 0$ is strongly minimal. So, strong minimality precludes the existence of invariant curves, \emph{but this is not sufficient.} For instance, the system 
\begin{eqnarray*} 
x' & = &  1 \\ 
y' & = & xy + \alpha
\end{eqnarray*}
is not strongly minimal, but when $\alpha \neq 0$ the system has no invariant curves.\footnote{Thanks to Maria Demina for this example.} It is easy to see that the system violates the transcendence criterion over the field $\m C (t)$ with the solution $x=t$ and $y$ a generic solution to $y'=ty+ \alpha$. 
\end{rem}

As already alluded to in the introduction (see also the discussion below), in $DCF_0$ strongly minimal sets determine, in a precise manner, the structure of all definable sets of finite order. Furthermore, establishing strong minimality of a definable set $X$ usually ensures that we have some control over the possible complexity of the structure of the set $X$. As an example, if $X$ is defined over $\mathbb{C}$, that is the differential equations involved are autonomous, then the following holds (cf. \cite[Section 2]{nagloo2011algebraic} and \cite[Section 5]{casale2020ax}).

\begin{fact}\label{autonomous}
Assume that a strongly minimal set $X$ is defined over $\mathbb{C}$ and that $ord(X)>1$. Then
\begin{enumerate}
    \item $X$ is orthogonal to $\mathbb{C}$.
    \item $X$ is geometrically trivial: for any differential field $K$ over which $X$ is defined, and for any $y_{1},..,y_{\ell}\in X$, denoting $\tilde{y}_i$ the tuple given by $y_i$ together with all its derivatives, if $(\tilde{y}_1,\ldots,\tilde{y}_{\ell})$ is algebraically dependent over $K$, then for some $i<j$, $\tilde{y}_{i}, \tilde{y}_{j}$ are algebraically dependent over $K$.
    \item If $Y$ is another strongly minimal set that is nonorthogonal to $X$, then it is non-weakly orthogonal to $X$.
\end{enumerate}
\end{fact}

Recall that if $X_1$ and $X_2$ are strongly minimal sets, we say that $X_1$ and $X_2$ are \emph{nonorthogonal} if there is some infinite definable relation $R\subset X_1\times X_2$ such that ${\pi_1}_{|R}$ and ${\pi_2}_{|R}$ are finite-to-one functions. Here for $i=1,2$, we use $\pi_i:X_1\times X_2\rightarrow X_i$ to denote the projections maps. Generally, even if the sets $X_1$ and $X_2$ are defined over some differential field $K$, it need not be the case that the finite-to-finite relation $R$ witnessing nonorthogonality is defined over $K$ (instead it will be defined over a differential field extension of $K$). We say that $X_1$ is non-\emph{weakly orthogonal} to $X_2$ if they are nonorthogonal and the relation $R\subset X_1\times X_2$ is defined over $K^{alg}$. 

\begin{rem}
Notice that in Fact \ref{autonomous}(2) we can replace ``$K$'' in the conclusion by ``$\mathbb{C}$'', that is one can state the conclusion as ``then for some $i<j$, $\tilde{y}_{i}, \tilde{y}_{j}$ are algebraically dependent over $\mathbb C$''. This follows using the non-weak orthogonality statement given in Fact \ref{autonomous}(3) (taking $Y=X$).
\end{rem}
In the next section, we will show that strong minimality holds in some special cases of equations of Li\'enard type. Since these equations are autonomous of order 2, it will then follows that all three conclusions of Fact \ref{autonomous} hold in those cases. This will allow us to make deeper analysis of the algebraic property of the solution sets.

It is worth mentioning that if a strongly minimal set is not necessarily defined over $\m C$, then there still is a strong classification result called the Zilber trichotomy theorem:

\begin{fact}[\cite{HrushovskiSokolovic},\cite{PillayZiegler}]\label{trichotomy} Let $X$ be a strongly minimal set. Then exactly one of the following holds:
\begin{enumerate}
\item $X$ is nonorthogonal to $\mathbb{C}$,
\item $X$ is nonorthogonal to the (unique) smallest Zariski-dense definable subgroup of a simple abelian variety $A$ which does not descend to $\mathbb{C}$, 
\item $X$ is geometrically trivial.
\end{enumerate}
\end{fact}

Notice that nonorthogonality to the constants is simply a version of algebraic integrability after base change. We will now discuss several other variations of this notion but first need to say a few words about ``types'' and ``forking'' in $DCF_0$.

Let $K$ be a differential field and ${y}$ a tuple of elements from $\mathcal{U}$, the type of ${y}$ over $K$, denoted $\tp({y}/K)$, is the set of all $L_{\delta}$-formulas with parameters from $K$ that ${y}$ satisfies. It is not hard to see that the set $I_{p}=\{f\in K\{\overline{X}\}: f(\overline{X})=0\in p\}=\{f\in K\{\overline{X}\}: f({y})=0\}$ is a differential prime ideal in $K\{\overline{X}\}=K[\overline{X},\overline{X}',\ldots]$, where $p=\tp({y}/K)$. Indeed, by quantifier elimination, the map $p\mapsto I_p$ is a bijection between the set of complete types over $K$ and differential prime ideals in $K\{\overline{X}\}$. Therefore in what follows there is no harm to think of $p=tp({y}/K)$ as the ideal $I_{p}$. If $X$ is a definable set over $K$, then by the (generic) type of $X$ over $K$ we simply mean $\tp(y/K)$ for $y\in X$ generic over $K$. We say that a complete type\footnote{So $p=tp(y/K)$ for some tuple $y$ from $\mathcal{U}$.} $p$ over a differential field $K$ is of finite rank (or order) if it is the generic type of some definable set over $K$ of finite order.

\begin{defn}
Let $K$ be a differential field and ${y}$ a tuple of elements from $\mathcal{U}$. Let $F$ be a differential field extension of $K$. We say that $\tp(y/F)$ {\em is a nonforking extension} of $\tp(y/K)$ if $K\gen{y}$ is algebraically disjoint from $F$ over $K$, i.e., if $y_1,\ldots,y_k\in K\gen{y}$ are algebraically independent over $K$ then they are algebraically independent over $F$. Otherwise, we say that $\tp(y/F)$ is a forking extension of $\tp(y/K)$ or that $\tp(y/F)$ forks over $K$.
\end{defn}
It is not hard to see from the definition that $\tp(y/K^{alg})$ is always a nonforking extension of $\tp(y/K)$. A complete type $p=\tp(y/K)$ over a differential field $K$, is said to be {\em stationary} if  $\tp(y/K^{alg})$ is its unique nonforking extension, i.e., whenever $z$ is another realization of $p$ (so $\tp(y/K)=\tp(z/K)$), then $z$ is also a realization of $\tp(y/K^{alg})$ (so $\tp(y/K^{alg})=\tp(z/K^{alg})$). We say that it is {\em minimal} if it is not algebraic and all its forking extensions are algebraic, that is if $q=tp(y/F)$ is a forking extension of $p$, where $F\supseteq K$, then $y\in F^{alg}$. If $X$ is strongly minimal and $p$ is its generic type, then if follows that $p$ is minimal. 

Using forking, one obtain a well-defined notion of independence as follows: Let $K\subseteq F$ be differential fields and ${y}$ a tuple of elements from $\mathcal{U}$. We say that $y$ is {\em independent} from $F$ over $K$ and write $y\ind_{K}{F}$, if $\tp(y/F)$ is a nonforking extension of $\tp(y/K)$. We now give the first variation of nonorthogonality to the constants.

\begin{defn}
A complete type $p$ over a differential field $K$ is said to be {\em internal to $\m C$} if there is some differential field extension $F\supseteq K$ such that for every realisation $y$ of $p$ there is a tuple $c_1,\ldots,c_k$ from $\m C$ such that $y\in F(c_1,\ldots,c_k)$.
\end{defn}

\begin{fact}\cite[Lemma 10.1.3-4]{tent2012course}\label{internality}
\begin{enumerate}
\item A complete type $p$ over a differential field $K$ is internal to $\m C$ if and only if there is some differential field extension $F\supseteq K$ and some realisation $y$ of $p$ such that $y\in F(\m C)$ and $y\ind_{K}{F}$.
\item A definable set $X$ is internal to $\m C$ if and only if there is a definable surjection from $\m C^n$ (for some $n\in\mathbb{N}$) onto $X$.
\end{enumerate}
\end{fact}
Using Fact \ref{internality}(2) it is not hard to see that homogeneous linear differential equations are internal to $\m C$. Indeed in this case, the solution set is simply a $\m C$-vector space $V$. If $(v_1,\ldots v_k)$ is a basis for $V$, then the map $f(x_1,\ldots x_k)=\sum_{i=1}^kx_iv_i$ is the surjective map $\m C^n\rightarrow V$ witnessing that $V$ is internal to $\m C$. Clearly, Fact \ref{internality}(2) also shows that internality to the constants is closely related to the notion of algebraic integrability (i.e. enough independent first integrals). We also have a more general but closely related notion of analysability in the constants:
\begin{defn}
\sloppy Let $y$ be a tuple from $\mathcal{U}$ and $K$ a differential field. We say that $\tp(y/K)$ is {\em analysable in the constants} if there is a sequence  $(y_{0},\ldots,y_{n})$ such that 
\begin{itemize}
    \item $y\in K\gen{y_{0},y_{1},\ldots,y_{n}}^{alg}$ and 
    \item for each $i$, either $y_{i}\in K\gen{y_{0},\ldots,y_{i-1}}^{alg}$ or $tp(y_{i}/K\gen{y_{0},\ldots,y_{i-1}})$ is stationary and internal to $\m C$.
\end{itemize}
\end{defn}
\sloppy It follows that if $\tp(y/K)$ is analysable in the constants, then the sequence $(y_{0},\ldots,y_{n})$ in the definition above can be chosen to be from $K\gen{y}$. Furthermore, it follows that analysability of $p$ in the constants is equivalent to the condition that every extension of $p$ is nonorthogonal to $\m C$. Differential equations that have Liouvillian solutions provide the most studied example of equations that are analysable in the constants. We will say quite a bit more in Section \ref{integrabilityandnot}.
Let us now turn our attention to the semi-minimal analysis of complete types, a notion which has been mentioned a few times in the introduction. 
\begin{defn} \label{semiminimal}
Let $p$ be a complete stationary type over  a differential field $K$. Then $p$ is said to be {\em semiminimal} if there is some differential field extension $F\supseteq K$ , some $z$ realising the nonforking extension of $p$ to $F$ and $z_1,\ldots,z_n$ each of whose type over $F$ is minimal and such that $z\in F\gen{z_1,\ldots,z_n}$.
\end{defn}
Semiminimal (and hence minimal) types are the building block all finite rank types in $DCF_0$ via the following construction

\begin{defn} \label{semimin}
Let $p=\tp(y/K)$ be a complete type over a differential field $K$. A \emph{semiminimal analysis of $p$} is a sequence $(y_0, \ldots,y_n)$ such that 
\begin{itemize}
\item $y\in K\gen{y_n}$, 
\item for each $i$, $y_i \in K\gen{y_{i+1}}$,
\item for each $i$, $\tp(y_{i+1} /K\gen{y_i})$ is semiminimal. 
\end{itemize} 
\end{defn}
The following is a fundamental result and is obtained by putting together Lemma 2.5.1 in \cite{GST} and Lemma 1.8 in \cite{BUECHLER2008135} (See aslo Proposition 5.9 and 5.12 in \cite{PillayNotes}).
\begin{fact}\label{semifact}
Every complete type of finite rank in $DCF_0$ has a semiminimal analysis.
\end{fact}
Finally, recall that for a field $K$, we denote by $K\left(\left(X\right)\right)$ the field of formal Laurent series in variable $X$, while $K\gen{\gen{X}}$ denotes the field of formal Puiseux series, i.e., the field $\bigcup_{d\in{\mathbb{N}}}K\left(\left(X^{1/d}\right)\right)$. It is well know that if $K$ is an algebraically closed field of characteristic zero, then so is $K\gen{\gen{X}}$ (cf. \cite[Corollary 13.15]{Eisenbud}). 

Puiseux series traditionally appear in the study of algebraic solutions of differential equations, however they have also been used by Nishioka (cf. \cite{nishioka1990painleve} and \cite{Nishioka2}) in his work around proving transcendence results for solutions of some classical differential equations. Inspired by those ideas, Nagloo \cite{nagloo2015geometric} and Casale, Freitag and Nagloo \cite{casale2020ax} have also use these techniques to study model theoretic and transcendence properties of solutions of well-known differential equations generalizing the results of Nishioka. In a different direction, Le\'on-S\'anchez and Tressl \cite{Leonlarge} also used Puiseux series in their work on differentially large fields. We will make use of Puiseux series in our proof of strong minimality of special cases of equations of Li\'enard type.

\section{Strong minimality} \label{strminsection}
The set of solutions of the equation $$z z'' = z',\;\;\;\;z'\neq 0$$ in a differentially closed field of characteristic zero were shown by Poizat (see \cite{MMP} for an exposition) to be strongly minimal. %{\color{red}RN: sadly we did not define RM and dM \st{of Morley rank one and Morley degree two; there is one infinite differential subvariety defined by $z'=0$}}. 
Poizat's method of proof relies in an essential way on the specific form of the equation being extremely simple.\footnote{The proof is direct; taking an arbitrary differential polynomial $p(z)$ of order one, if the polynomial determines a subvariety, it must be that the vanishing of $p(z)$ implies the vanishing of $zz''-z'$. Considering  $z\delta (p(z)) $ one can apply the relation $zz''=z'$ to obtain a new differential polynomial $q(z)$ of order one such that the vanishing of $p(z)$ implies the vanishing of $q(z)$. It follows that $p(z)$ must divide $q(z)$, and this fact can be used to show that $p(z)$ itself must be of a very restrictive form. One ultimately shows that $p(z) = z'.$} A similar but more complicated variant of the strategy of Poizat was employed in Kolchin's proof of the strong minimality of the first Painlev\'e equation (originally in an unpublished letter from Kolchin to Wood); an exposition appears in \cite{MMP}. In \cite[Chapter 9]{freitag2012model}, another elaboration of the above strategy was employed to show that the set defined by $$zz'''-z''=0 , \, \text{ and  } \, \, z'' \neq 0$$
is strongly minimal. 

In \cite{brestovski1989algebraic}, Brestovski generalized Poizat's theorem to include equations of the form: 
$$z''= z' \left(\frac{B - f_z z' -g_z}{fA} \right),\;\;\;\;z'\neq 0$$ for polynomials $f,g,A,B$ over $\m C$ satisfying very specific conditions.\footnote{When $f, g$ are constant, $B=1, A=z$ the theorem yield's Poizat's result and these choices satisfy Brestovski's assumptions.  The assumptions in Brestovski's theorem are calibrated just so that the strategy of Poizat can be successfully carried out. A complete characterization of strong minimality via this method seems unlikely, due to the complexity of the calculations which appear in the course of the proof in \cite{brestovski1989algebraic}.} We are interested in the case that the derivatives of $z$ appear linearly in the equation (i.e. $f$ is a constant). Then Brestovski's family of equations becomes: 

\begin{equation}\tag{$\star$} \label{stareqn}
z''={z'}f(z),\;\;\;\;z'\neq 0
\end{equation}
where $f(z) \in \m C(z)$. In this case, we give a definitive characterization of the strong minimality:

%\begin{thm} Let $f,g,A,B \in \m C [z]$ where $B/A$ is a linear combination of logarithmic derivatives over $\m C$: $$\sum_{i=1}^n  c_i \frac{(h_i)_z }{h_i}$$ where the elements $c_i$ are linearly independent over $\m Q$ with $h_i \in \m C(z)$ nonconstant. Then the only order one subvariety of the differential equation $$z''= z' \left(\frac{B - f_z z' -g_z}{fA} \right)$$ is given by $z'=0.$ 
%\end{thm}

\begin{thm} \label{stminthm} The solution set of equation (\ref{stareqn}) is strongly minimal if and only if for all $g\in \m C(z)$, we have that $f(z)\neq\frac{d g}{dz}$.
\end{thm} 

\begin{proof} Clearly, if $f(z)=\frac{d g}{dz}$ for some $g\in \m C(z)$, then any solution of $z'=g(z)+c$, $c\in \m C$, is also a solution to $\frac{z''}{z'} = f(z)$. Hence the solution set of equation (\ref{stareqn}) is not strongly minimal and indeed has rank 2.

Now assume that $f(z)$ has partial fraction decomposition
$$f(z)=\frac{d g}{dz}+\sum_{i=1}^n{\frac{c_i}{z-a_i}}$$ where the $a_i$'s are distinct and some $c_i\neq 0$. Without loss of generality assume $c_1\neq 0$. Then $f(z)$ has a nonzero residue at $a_1$. Considering the change of variable $z\mapsto z-a_1$ we may assume that $f(z)$ has a nonzero residue at $0$.

Arguing by contradiction, let us assume that the solution set of equation (\ref{stareqn}) is not strongly minimal. Then for some $K$, a finitely generated differential field extending $\m C$\footnote{Formally, we work with $\mathcal{C}\subset \m C$ a subfield finitely generated over $\m Q$ by the coefficients of the equation.} with derivation $\delta$, and $y$ a solution of (\ref{stareqn}) we have that $u = \delta(y) \in K(y)^{alg}.$

\sloppy We can think of $u$ as living in the field of Puiseux series $K^{alg}\gen{\gen{y}}$ with the usual valuation $v$ and the derivation
$$\delta \left(\sum a_iy^i\right)=\sum \delta(a_i)y^i+\left(\sum ia_iy^{i-1}\right)\delta(y).$$

So $$u=\sum_{i=0}^{\infty} a_iy^{r+\frac{i}{m}},$$ where $v(u)=r$ and $m$ is the ramification exponent. Differentiating we get
$$\delta(u)=\sum_{i=0}^{\infty} \delta(a_i)y^{r+\frac{i}{m}}+u\left(\sum_{i=0}^{\infty} (r+\frac{i}{m})a_iy^{r+\frac{i}{m}-1}\right).$$
Since $$v\left(\sum_{i=0}^{\infty} \delta(a_i)y^{r+\frac{i}{m}}\right)\geq r,$$ we have that 

$$\frac{\delta(u)}{u}=\alpha +\sum_{i=0}^{\infty} (r+\frac{i}{m})a_iy^{r+\frac{i}{m}-1},$$
where $v(\alpha)\geq 0$. The right hand side of this equation is equal to $f(y)$ and so there should be a nonzero residue. But the coefficient of $y^{-1}$ on the right hand side is 0. This is a contradiction.
\end{proof} 
Since Equation (\ref{stareqn}) has constant coefficients, it follows from Theorem \ref{stminthm} and Fact \ref{autonomous}(2) (see \cite[Proposition 5.8]{casale2020ax} for a proof) that:
\begin{cor} \label{triviality} The solution set of equation (\ref{stareqn}) for $f(z)$ not the derivative of any rational function is geometrically trivial.
\end{cor}
The previous corollary already gives strong restrictions on the possible algebraic relations between solutions of Equation (\ref{stareqn}), but sections \ref{formssection} and \ref{algrelsection} are devoted to giving a complete classification. Following this, we turn to similar questions in the case that $f(z)$ is the derivative of a rational function. Before we do so let us describe the connection between Theorem \ref{stminthm} and (non)integrability of equations of Li\'enard type.
%%%%%%%%%%%%%%%%%%%%%%%%%%%%%%%%%%%%%%%%%

\section{Solutions and integrability} \label{integrabilityandnot}
Equations of the form: 
\begin{equation} \label{Lie}
    x'' (t) + f(x) x' (t) + g(x) =0,
\end{equation}
for $f(x),g(x)$ rational functions have their origins in the work of Li\'enard \cite{Lie1, Lie2} and have important applications in numerous scientific areas. For instance, the solutions can be used to model oscillating circuits; see page 2 of \cite{harko2014class} for numerous references. Numerous recent works are devoted to giving explicit solutions or first integrals of Equation \ref{Lie} in special cases or showing that none can be expressed in terms of special functions in some class (e.g. Liouvillian, elementary). In this section, we first point out some general connections between solutions in special classes of solutions, first integrals, and the model theoretic notions we study. Following this, we describe some existing results for Li\'enard equations then give some results based on model theoretic ideas and our work in Section \ref{strminsection}. 

\subsection{Special classes of solutions} 
In this section, we give results connecting our model theoretic notions to several classically studied classes of solutions. 
\begin{defn}
Let $(F, \Delta)$ be a differential field (generally we are interested in the case $F=\m C(x,y)$ with the derivations $\frac{d}{dx}, \frac{d}{dy}$). We say that   
a finitely generated differential field extension $(K, \Delta)$ of $F$ is \emph{elementary} if there is a tower of differential field extensions $F=F_0 \subset F_1 \ldots , \subset F_n = K$ such that for all $i=1, \ldots n$ we have that $F_i = F_{i-1} (\alpha ) $ where $\alpha$ is such that: 
\begin{enumerate}
    \item $\delta \alpha  = \delta f /f$ for some $f \in F_{i-1}$ and for all $\delta \in \Delta$ \emph{or} 
    \item $\delta \alpha /\alpha =\delta f$ for some $f \in F_{i-1}$ and for all $\delta \in \Delta$ \emph{or}
    \item $\alpha \in F_{i-1}^{alg}.$
\end{enumerate}
\end{defn}

The class of Liouvillian functions is more general than the class of elementary functions: 
\begin{defn}
Let $(F, \Delta)$ be a differential field. We say that   
a finitely generated differential field extension $(K, \Delta)$ of $F$ is \emph{Liouvillian} if there is a tower of differential field extensions $F=F_0 \subset F_1 \ldots , \subset F_n = K$ such that for all $i=1, \ldots n$ we have that $F_i = F_{i-1} (\alpha ) $ where $\alpha$ is such that: 
\begin{enumerate}
    \item $\delta \alpha  \in F_{i-1}$ for all $\delta \in \Delta$ \emph{or} 
    \item $\delta \alpha /\alpha  \in F_{i-1}$ for all $\delta \in \Delta$ \emph{or}
    \item $\alpha \in F_{i-1}^{alg}.$
\end{enumerate}
\end{defn}

We next give several more special classes of functions generalizing Liouvillian and elementary functions. 

\begin{defn}\footnote{The notion of a Pfaffian function is most commonly defined for a real-valued function of a real variable, but we formulate the complex analog as well which fits more naturally with the results of this paper. Both notions are closely connected to model theoretic notions from the theory of differentially closed fields. See \cite{freitag2021not}.} Let $f_1, \ldots , f_l $ be complex analytic functions on some domain $U \subseteq \m C^n$. We will call $(f_1, \ldots , f_l)$ a \emph{$\m C $-Pfaffian chain} if there are polynomials $p_{ij}(u_1, \ldots , u_n , v_1, \ldots, v_i )$ with coefficients in $\m C$ such that $$\pd{f_i}{x_j}= p_{ij} \left( \bar x, f_1 ( \bar x), \ldots , f_i (\bar x ) \right)$$ 
for $1 \leq i \leq l$ and $1 \leq j \leq n.$
We call a function \emph{$\m C$-Pfaffian} if it can be written as a polynomial (coefficients in $\m C$) in the functions of some $\m C$-Pfaffian chain.
\end{defn}

Finally, we come to the most general notion we consider, a condition that was developed by Nishioka \cite{nishioka1990painleve, nishiokaII}:

\begin{defn}
Let $y$ be differentially algebraic over a differential field $k$. We say $a$ is \emph{$r$-reducible over $k$} if there exists a finite chain of $k$-finitely generated differential field extensions, $$k=R_0 \subset R_1 \subset \ldots R_m$$ such that $a \in R_m$ and $\trdeg{R_i/R_{i-1}} \leq r.$ 
\end{defn}

\begin{thm}
If $X$ is a strongly minimal differential equation of order $n$ defined over a finitely generated differential field $K$, then any nonalgebraic solution $f$ of $X$ is not $d$-reducible for any $d<n.$ It also follows that $f$ is not Pfaffian, Liouvillian, or elementary.
\end{thm}

\begin{proof}
Recall, from Remark \ref{SMLienard} that the zero set of our differential equation $X$ with coefficients in a differential field $K$ is strongly minimal if and only if (1) the equation is irreducible over $K^{alg}$ (as a polynomial in several variables) and (2) given any solution $f$ of $X$ and \emph{any differential field extension $F$ of $K$}, $$\text{trdeg}_F F\langle f \rangle =  \text{trdeg}_K K \langle f \rangle \text{ or } 0.$$ If $f$ were $d$-reducible for $d<n$, as witnessed by some chain $K=R_0 \subset R_1 \subset \ldots R_m$, then we can assume that $f$ is transcendental over $R_{m_1}$ for some $m_1 \leq m$ and algebraic over $R_{m_1}$. But then the differential field $R_{m_1}$ has the property that $\text{trdeg}_{R_{m_1}} \left( R_{m_1} \langle f \rangle \right) \leq d <n,$ contradicting strong minimality of $X$. Of course, each of the classes Pfaffian, Liouvillian, and elementary are $1$-reducible, so $f$ can not be in any of these classes either. 
\end{proof}

Assuming a weaker model theoretic notion about $X$ allows one to rule out Liouvillian solutions, but not Pfaffian solutions:  

\begin{thm} \label{fact} Let $X$ be a differential equation of order $n$ defined over a finitely generated differential field $K$. 
Suppose the generic type of $X$ is not analyzable in the constants; then the generic solution of $X$ is not Liouvillian. 

Suppose further that $X$ is orthogonal to the constants. Then any nonalgebraic solution $f$ of $X$ is not Liouvillian. 
\end{thm}

\begin{proof}
Recall from Fact \ref{semifact} that every finite rank type has a semiminimal analysis. The extensions appearing in the definition of $f$ being Liouvillian are either algebraic or generated by the generic solution of an order one linear differential equation. The type of the generator of this extension is \emph{internal to the constants}\footnote{See Fact \ref{trichotomy}.} over the previous field in the tower, so the type of $f$ over $K$ is analyzable in the constants. 

If $X$ (as a definable set) is orthogonal to the constants, then any type $q$ in $X$ not algebraic over $k$ has the property that $q$ is orthogonal to the constants. This implies $q$ is not analyzable in the constants, so any realization of $q$ is not Liouvillian. 
\end{proof}

Non-analyzability or even orthogonality to the constants does not rule out the more general Pfaffian or $d$-reducible solutions as above. The connection between integrability in Liouvillian or elementary terms and our model theoretic notions is more subtle than the connection to the existence of solutions, as we explain in the next subsection.

\subsection{Notions of integrability}
We will begin by describing some general notions around integrability. Consider a system of autonomous differential equations \begin{equation} 
\bar x ' = P(\bar x ) \label{vecfield}
\end{equation} where $P=(P_1,\dots,P_n)$ are polynomial or rational functions in the variables $\bar x = (x_1, \ldots , x_n) $ with coefficients in $\m C^n$. 

%A \emph{first integral} of the system is a function of $\bar x$ which is constant along solution curves of the system, i.e., $F:{\m C}^n\rightarrow {\m C}$ with $\sum_{i=1}^n P_i(\bar x){\partial F\over \partial x_i}=0$. 

%{\color{blue}[one needs to exclude constant functions from the previous definition, also they do not need to be defined on the whole affine space $\mathbb{C}^n$, here is a suggestion] 

A \emph{first integral} of the system is a non constant meromorphic function of $\bar x$ which is constant along solution curves of the system, i.e., $F: U \subset {\mathbb C}^n\rightarrow {\mathbb C}$ defined on some non-empty analytic open set $U$ of $\mathbb{C}^n$ with $$\sum_{i=1}^n P_i(\bar x){\partial F\over \partial x_i}=0.$$
Meromorphic (and even holomorphic) first integrals always exist in an analytic neighborhood of a non-singular point of the equation; furthermore if $F$ is a first integral of the system on some open set $U$ then it is a first integral on any open set $U \subset V$ where $F$ can be analytically continued.  In particular, if $F$ is a rational function then the open set $U$ can be taken to be the Zariski-open set of $\mathbb{C}^n$ where $F$ is well-defined.%}

Usually, one is interested in first integrals from various special classes of functions. For instance, a \emph{Darboux integral} \cite{MR3563433} of the system is one of the special form: 
$$f_1 ( \bar x ) ^{r_1} \ldots f_k ( \bar x ) ^{r_k} e^{h(x) /g(x)} $$ for polynomials $f_i,g,h$ and $r_j \in \m C$. 

Associated with the polynomials $P(\bar x ) =  (P_1 (\bar x ), \ldots , P_n (\bar x ))$ is the vector field $$\tau _P := P_1( \bar x ) \pd{}{x_1} + \ldots + P_n (\bar x ) \pd{}{x_n}.$$ 
A \emph{Darboux polynomial} of the system is $f ( \bar x ) \in \m C[ \bar x ]$ such that $\tau _P (f) $ divides $f$. This condition is equivalent to the zero set of $f$ being an invariant algebraic hypersurface for the vector field $\tau$. The connection to integrability is given by results originally due to Darboux and Jouanolou, see \cite[Theorem 3]{MR2902728}.

\begin{fact}
%{\color{red}Should this be Fact 4.7 with references}
Suppose that a polynomial vector field $\tau$ of degree at most $d$ has irreducible invariant hypersurfaces given by the zero set of $f_i$ for $i=1, \ldots k$ and suppose that the $f_i$ are relatively prime. Then: 
\begin{enumerate}
    \item If $k \geq \binom{n+d-1}{n}+1$ then $\tau $ has a Darboux integral. 
    \item If $k \geq \binom{n+d-1}{n}+n$ then $\tau $ has a rational first integral. 
\end{enumerate}
\end{fact}

In model theoretic terms, even in the nonautonomous case, there is a close connection between co-order one differential subvarieties of a differential algebraic variety and nonorthogonality to the constants, see \cite{freitag2017finiteness}. Of course, the relation to the previous section is: strong minimality of a second order (or higher) system of differential equations implies that the system has no Darboux polynomials. In fact strong minimality and the other model theoretic notions we study go a good deal further, but as we will see, our model theoretic notions are more closely connected to the existence of solutions in various special classes rather than integrability in those classes.

\begin{defn}
We call a first integral $F$ \emph{elementary (Liouvillian)} if $F$ is an elementary (Liouvillian) function.\footnote{Any of the special classes of functions we mention in the previous subsection might be used to develop notions of integrability, but to our knowledge there is no development of integrability in terms of Pfaffian or $r$-reducible functions.}
\end{defn}

We first remark that one can reduce the study of algebraic integrals to the study of rational integrals. 

\begin{lem}
Let $X$ be a vector field on some algebraic variety over $\mathbb{C}$.  If $X$ has an algebraic first integral then $X$ has a rational first integral
\end{lem}

\begin{proof}
We denote by $V$ the algebraic variety supporting $X$ and by $\delta$ the derivation induced by $X$ on $\mathbb{C}(V)$.  First remark that since $\delta$ extends uniquely to a derivation $\overline{\delta}$ on $\mathbb{C}(V)^{alg}$, we have 
$$ \overline{\delta} \circ \sigma = \sigma \circ \overline{\delta} \text{ for all } \sigma \in Gal(\mathbb{C}(V)^{alg}/\mathbb{C}(V))$$
as $\sigma^{-1} \circ \overline{\delta} \circ \sigma$ is another derivation on $\mathbb{C}(V)^{alg}$ extending $\delta$.

Assume now that $X$ has no rational first integrals and consider $f \in \mathbb{C}(V)^{alg}$ such that $\overline{\delta}(f) = 0$.  Then by the remark above,  we also have $\overline{\delta}(\sigma(f)) = 0$ for all $\sigma \in Gal(\mathbb{C}(V)^{alg}/\mathbb{C}(V))$.  Hence the coefficients $a_1,\ldots, a_r \in \mathbb{C}(V)$ of the minimal polynomial of $f$ over $\mathbb{C}(V)$ satisfy $\delta(a_i) = \overline{\delta}(a_i) = 0$ and therefore by assumption $a_1,\ldots, a_r \in \mathbb{C}$.  Since $\mathbb{C}$ is algebraically closed,  we conclude that $f \in \mathbb{C}$ and that $X$ does not have any algebraic integral either.
\end{proof}

\begin{thm}
Let $X$ be a vector field on some algebraic variety over $\m C$. If $X$ has an algebraic first integral, then $X$ is not orthogonal to the constants. 
\end{thm}
\begin{proof}
An \emph{algebraic} first integral $f$ gives a map from the solution set of $X$ to $\m C$ as $f$ is constant on solutions. When $f$ is algebraic, this yields a definable map from $X$ to $\m C$, implying $X$ is nonorthogonal to $\m C$.
\end{proof}

For the remainder of the section we work with more general first integrals, but we will assume the differential equation we work with, $X$, is given by a planar vector field with coefficients in $\m C$.

\begin{thm}
Let $X$ as above be an order two differential equation given by a rational planar vector field over $\m C$. If $X$ has an elementary first integral, then $X$ has an integrating factor of the form: 
$$\Pi (C_i)^{p_i}$$ for polynomials $C_i$ and integers $p_i$.
If $X$ is strongly minimal then all of the $C_i$ must be poles of the vector field. If $X$ is regular and strongly minimal, then $X$ has no elementary first integral. 
\end{thm}
\begin{proof}
If the system $X$ has an elementary first integral, results of \cite{prelle1983elementary} show that the integrating factor is of the form 
$$\Pi (C_i)^{p_i}$$ for polynomials $C_i$ and integers $p_i$.\footnote{Technically, \cite{prelle1983elementary} works in the setting of regular vector fields, but an easy argument shows that the results apply to rational vector fields as well; see page 8 of \cite{duarte2009finding}.}

It follows that if the $C_i$ are not poles of the vector field, then the system has nontrivial invariant algebraic curves (an explanation of this can be found in various place, e.g. the second page of \cite{christopher1999liouvillian} following the statement of the main theorem). Strongly minimal systems have no invariant curves, and regular systems have no poles.
\end{proof}

The connection between Liouvillian first integrals and strong minimality is more subtle, but we can say something about the form of the integrating factor:

\begin{thm} \label{Liouv}
If $X$ is a strongly minimal planar vector field with coefficients in $\m C$, then $X$ has a Liouvillian first integral if and only if $X$ has an integrating factor of the form $\Pi (C_i)^{p_i} e^{C/D}$ for polynomials $C_i, D$ which are poles of the vector field and $C$ a polynomial. If $X$ is a strongly minimal regular vector field, then if $X$ has a Liouvillian first integral, it has an integrating factor of the form $e^C$. 
\end{thm}

\begin{proof}
By results of Singer \cite{MR1062869} and Christopher \cite[Theorem 2]{christopher1999liouvillian}, if there is a Liouvillian first integral of $X$, then there is an integrating factor of the form: 
$$e^{C/D} \cdot \Pi (C_i)^{p_i}$$ 
where $C,D, C_i$ are polynomial functions of the two variables of the system. Their proofs take place in the regular setting, but can be adapted to rational vector fields; see \cite{duarte2009finding}. The zero sets of the $C_i$ and the zero set of $D$ give invariant algebraic curves for the vector field as long as they are not poles of the vector field $X$, contradicting strong minimality. 
\end{proof}

We now describe two examples. The first ones shows that Liouvillian integrability does not in general imply the existence of invariant algebraic curves. 

\begin{exam}
Consider the system
\begin{equation} \label{demi}
\begin{array}{r@{}l}
x' &{} =  1 \\ 
y' &{} =  xy + \alpha
\end{array}
\end{equation}
where $\alpha \neq 0$. The system has integrating factor $e^{\frac{-x^2}{2}}$, so the system has a Liouvillian first integral, but no invariant algebraic curve.
\end{exam}

Notice that the system \ref{demi} is not strongly minimal and that more precisely the solutions of this system are all Liouvillian. On the other hand, Rosenlicht constructed examples of order two equations having a Liouvillian first integral but no nonconstant Liouvillian solution \cite[introduction]{rosenlicht1969explicit} \cite[Proposition 3]{MR3563433}. Our second example shows that there exist order two equations having a Liouvillian first integral but \emph{no Pfaffian solution}.  

\begin{exam}
Consider the vector field associated with the Poizat equation which originally motivated our work: 
\begin{equation} \label{poi1}
\begin{array}{r@{}l}
x' &{} =  y \\ 
y' &{} =  y/x
\end{array}
\end{equation}
Note that the first integrals of the system are unaffected by multiplying both rational functions by $x$ to clear the denominator of the second equation. One then obtains the system: 
\begin{equation} \label{poi2}
\begin{array}{r@{}l}
x' &{} =  xy \\ 
y' &{} =  y
\end{array}
\end{equation}
It is easy to check that the function $H(x,y) = \frac{e^y}{x}$ is a first integral of this second system (hence of the first one too), which has two invariant curves given by $x=0$ and $y=0$. It is also easy to see that the generic solution of system \ref{poi2} is not strongly minimal or orthogonal to the constants (it is 2-step analyzable in the constants and has a Liouvillian generic solution), while the generic solution of system \ref{poi1} is strongly minimal by the arguments of the previous section. So system \ref{poi1} is a system with a Liouvillian first integral but no Pfaffian solution.
\end{exam}
Furthermore, this example illustrates the following observation of independent interest: transformations which scale both coordinates of the vector field by some polynomial
\begin{itemize} 
\item preserve first integrals,
\item do not preserve the model theoretic notions we study (e.g. strong minimality, orthogonality to the constants), 
\item do not preserve the property of the system having Liouvillian solutions. 
\end{itemize} 

The examples given above also show that the Theorem \ref{Liouv} can not be improved to give a direct connection between strong minimality and the existence of Liouvillian first integrals, at least not in complete generality. However, in the case that one can rule out an exponential integrating factor by some other argument, one can use strong minimality to show that no Liouvillian first integral exists. For instance, an argument ruling out exponential integrating factors in the case of certain Li\'enard equations is contained in \cite[Section 2]{MR3563433}.

\subsection{Overview of previous results for Li\'enard equations} \label{prevlie}

Equation \ref{Lie} is equivalently expressed by the vector field on $
\m A ^2$: 
\begin{equation} \label{Lie1}
\begin{array}{r@{}l}
   x' &{}= y \\
   y' &{}= -f(x) y - g(x) 
\end{array}
\end{equation} %Equations of this form have been the subject of intense work from a variety of perspectives due to their important applications in a variety of fields. In this subsection, we restrict to recounting some results from the point of view of integrability. 

The study of algebraic solutions of Equation \ref{Lie1} seems to begin with Odani \cite{ODANI1995146}, who shows that Equation \ref{Lie1} has no invariant algebraic curves when $f,g \neq 0$, $\deg (f) \geq \deg (g)$ and $g /f$ is nonconstant. Numerous authors attempted to generalize Odani's results on invariant curves \cite{MR1433130, MR2430656}. Many recent works utilize the results of Odani and generalizations to characterize Liouvillian first integrals of Li\'enard equations in various special cases \cite{10.2307/20764280, LlibreValls+2013+825+835, Cheze2021, MR3573730, MR3808495, MR4190110, demina2021method}. Many of the special cases considered make assumptions about the degrees of $f(x), g(x)$ in equation \ref{Lie1}, while others make detailed assumptions not unlike the criteria employed by Brestovski \cite{brestovski1989algebraic}. Demina \cite{demina2021integrability} has recently completely classified the systems \ref{Lie1} which have Liouvillian first integrals for polynomial $f,g$. 

Explicit exact solutions (all Liouvillian) for the Equation \ref{Lie1} in very special cases are the subject of many additional papers in the literature \cite{feng2001algebraic, feng2002explicit, feng2004exact, harko2014class, KONG1995301}. Our results in the next subsection show in numerous wide-ranging cases Equation \ref{Lie1} has no Liouvillian solutions, so formulas for explicit exact solutions such as those of \cite{feng2001algebraic, feng2002explicit, feng2004exact, harko2014class, KONG1995301} do not exist.

Numerous other order two systems of differential equations can be transformed analytically or algebraically to solutions of a system in the form of Equation \ref{Lie1}. In most cases, it is apparent that the transformations preserve the property of being Liouvillian. For instance, this applies to the transformations in Propositions 2 and 3 of \cite{gine2011weierstrass}. There it is shown that the solutions of the system \begin{eqnarray*}
x' & = & f_0(x) - f_1 (x) y , \\ 
y' & = & g_0 ( x) + g_1 (x) y + g_2 (x ) y^n 
\end{eqnarray*} 
can be transformed to solutions of the Li\'enard family \ref{Lie} by means of the transformation $$Y = (f_0(x) - f_1(x) y ) e^{ \int_ {0} ^ x \left( g_2 (\tau ) - f_1' (\tau )/ f_1(\tau ) \right) d \tau }.$$ It is easy to see that when the functions appearing in the system are Liouvillian, this analytic transformation preserves the property of solutions being Liouvillian. Similar more complicated analytic transformations have been developed for various particular order two systems of higher degree (e.g. Proposition 3 of \cite{gine2011weierstrass}). There are numerous additional works showing particular systems can be transformed into equations of Li\'enard form (see e.g. \cite{transformLien} or the references of \cite{gine2011weierstrass}). %Invariably, to the knowledge of the authors, the transformations used on these particular systems in the existing literature preserve the property of solutions being Liouvillian. 

\subsection{Solutions of Li\'enard type equations} \label{nonintLie}

%\begin{thm} Let $X \mapsto Y(\m C)$ be a family of differential equations over $\m Q$ such that all of the fibers $X_a$ and $Y$ are absolutely irreducible. If, for generic $\alpha \in Y(\m C)$, the generic type of $X_ \alpha$ is non-orthogonal to the constants, then for any $a \in Y( \m C)$, the generic type of $X_a$ is non-orthogonal to the constants. 
%\end{thm}
\begin{thm} \label{nonlou} \cite[Theorem C]{jaoui2020corps} Let $k$ be a countable field of characteristic $0$, let $S$ be a smooth irreducible algebraic variety over $k$ and let $\pi: (\mathcal X,v) \rightarrow (S,0)$ be a smooth family of autonomous differential equations indexed by $S$ defined over $k$.   Assume that all the fibres of $\pi$ are absolutely irreducible and that
\begin{center} 
$(O):$ for some $s_0 \in S(k)$,  the generic type of the fibre $(\mathcal X,v)_{s_0}:= \pi^{-1}(s_0)$ is orthogonal to the constants.
\end{center}
Then for some/any realization $s \in S(\mathbb{C})$ of the generic type of $S$  over $k$, the generic type of $(\mc X,v)_s$ is also orthogonal to the constants.
\end{thm}

By Theorem \ref{fact} and the conclusion of Theorem \ref{nonlou}, when condition $(O)$ holds and the system $(\mc X,v)_s$ is two-dimensional, the system $(\mc X,v)_s$ has only finitely many Liouvillian solutions. Note that because the theorem only says that the \emph{generic type} is orthogonal to the constants, there might be finitely many other types of order one coming from the finitely many algebraic invariant curves.

We fix $k$ a countable field of characteristic $0$ (for example,  $k = \mathbb{Q}$). Set $S = \mathbb{A}^p$ the affine space of dimension $n$.  By an \emph{$k$-algebraic family of rational functions indexed by $S$},  we mean a rational function $g(s,z) \in k(S)(z)$. 

\begin{lem} 
Let $f(s,z) \in k(S)(z)$. There is a dense open set $S_0 \subset S$ such that $f(s,z) \in \mathbb{C}[S_0](z)$.
\end{lem}

\begin{proof}
Write $$f(s,z) = \frac {g(s,z)}{h(s,z)} = \frac {\sum a_i(s) z^i}{\sum_{i \geq 1} b_i(s) z^i + 1}$$
where the $a_i's$ and the $b_i's$ are in $k(S)$. Denote by $Z$ the proper closed subset of $S$ obtained as the finite union of the poles of the $a_i$'s and the $b_i$'s and set $S_0 = S \setminus Z$.
\end{proof}

\begin{cor} \label{orthcor}
Let $k$ be a countable field of characteristic $0$,  let $g(s,z) \in k(S)(z)$ be a $k$-algebraic family of rational functions indexed by $S = \mathbb{A}^p$ and let $f(z) \in k[z]$ be a rational function with at least one non-zero residue.  Assume that 
\begin{center}
for some $s_0 \in S(k)$, the rational function $g(s_0,z)$ is identically  equal to $0$.
\end{center} 
Then for every realization $s \in S(\mathbb{C})$ of the generic type of $S$ over $k$ , the generic type of 
$$ y'' + y'f(y) + g(s, y) = 0.$$
is orthogonal to the constants. 
\end{cor}
Notice that the conclusion is \emph{equivalent} to: the property 
\begin{center}
$O(s)$: the generic type of $ y'' + y'f(y) + g(s, y) = 0$ is orthogonal to the constants
\end{center}
holds on a set of full Lebesgue measure of the parameter space $S(\mathbb{C})$.

\begin{proof}
Without loss of generality,  we can replace $S$ by an open set $S_0$ such that $g \in \mathbb{C}[S_0](z)$:  since $S$ is irreducible,  so is $S_0$ and $s_0 \in S_0$.  Denote by $(z,z')$ the standard coordinates on $\mathbb{A}^2$,  $S$ the (finite) set of poles of $f(z)$ and by $U \subset \mathbb{A}^2$ the Zariski open set defined by
$$U = \mathbb{A}^2 \setminus (S \times \mathbb{A}^1)$$

Consider $\pi: \mathcal X = U \times S_0 \rightarrow S_0$ which is obviously smooth and with the notation of the previous lemma consider the closed subset $Z$ of $\mathcal X$ defined by: 
$$ 1 +  \sum_{i \geq 1} b_i(s)z^i = 0$$
describing the set of poles of $g(s,z)$ when $s$ varies in $S_0$.  Since the restriction of a smooth morphism is smooth,  the restriction of $\pi$ to the open set $\mathcal X_0 = \mathcal X \setminus Z$ 
$$\pi_0: \mathcal X_0 \rightarrow S_0.$$
 is also smooth. Moreover,  the fibres of $\pi_0$ are absolutely irreducible since the fibres of $\pi$ are absolutely irreducible and a dense open set of an absolutely irreducible variety is also absolutely irreducible. 
 
 Consider the vector field on $\mathcal{X}_0$ given in the coordinates $(z,z',s)$ by 
 $$ v(z,z',s) = z' \frac \partial {\partial z} + \Big(- z'f(z) - g(z,s)\Big) \frac \partial {\partial z'} + 0 \frac \partial {\partial s_1} + \ldots + 0 \frac \partial {\partial s_p} .$$
By definition, the vector field $v$ is tangent to the fibres of $\pi_0$ so that
$$\pi_0: (\mathcal X_0,v) \rightarrow (S_0,0)$$
is a morphism of $D$-varieties and it satisfies the ``geometric'' assumptions of Theorem C by the discussion above. 

\begin{claim} \label{theclaim1}
Let $s \in S_0(\mathbb{C})$ and denote by $(\mathcal X_0,v)_{s}:= \pi^{-1}(s)$.  There is a $k(s)$-definable bijection between $(\mathcal X_0,v)^\delta_{s}$ and the solution set of $y'' + y'f(y) + g(y,s) = 0.$ In particular, the generic type of one is interdefinable over $k(s)$ with the generic type of the other.
\end{claim}

Indeed, this is the standard correspondence between $D$-varieties and differential equations: the definable bijection is given by:
$$(z,z') \mapsto z$$

For $s_0 \in S_0$, we have shown that the definable set $y'' + y'f(y) = 0$ has Morley rank $1$ (and Morley degree $2$).  Hence the generic type of this equation --- the unique type $p \in S(k)$ of maximal order living on the solution set of this equation --- is a strongly minimal type of order $2$, hence orthogonal to the constants.  The claim above shows that $s = s_0$ satisfies the property $(O)$. 

By Theorem 1.1,  we conclude that for generic $s \in S_0(\mathbb{C})$ (equivalently, for generic $s \in S(\mathbb{C})$) the generic type of $(\mathcal X_0,v)_s$ is orthogonal to the constants.  Hence using the claim above in the other direction, we obtain that the generic type of 
$$y'' + y'f(y) + g(s,y) = 0$$
for generic values of $s \in S(\mathbb{C})$ over $k$.
\end{proof}

\begin{exam}
Let $a_0,\ldots, a_n,b_0,\ldots b_n \in \mathbb{C}$ be $\mathbb{Q}$-algebraically independent. Then the generic type of
\begin{equation} \label{ex1} y'' + \frac{y'}{y} + \frac{a_ny^n + a_{n-1}y^{n-1} + \ldots + a_0}{b_my^m + b_{m-1}y^{m-1} + \ldots + b_0} = 0 \end{equation}
is orthogonal to the constants. By Fact \ref{fact}, the generic solutions of this equation are not Liouvillian and more precisely, this equation has at most finitely many nonconstant Liouvillian solutions which are all supported by algebraic invariant curves of the equation.  
\end{exam}

\begin{exam}
Let $a \notin \mathbb{Q}^{alg}$ be a transcendental number and $g(y) \in \mathbb{Q}(y)$ arbitrary.  The generic type of
\begin{equation} \label{ex2} y'' + \frac{y'}{y} + ag(y)= 0 \end{equation}
is orthogonal to the constants.  By Fact \ref{fact}, the generic solutions of this equation are not Liouvillian and more precisely, this equation has at most finitely many nonconstant Liouvillian solutions which are all supported by algebraic invariant curves of the equation. 
\end{exam}

\begin{rem}
Systems satisfying condition $(O)$ from Theorem \ref{nonlou} yield wide classes of examples generalizing Equations \ref{ex1} \ref{ex2}. For instance, one can replace $\frac{1}{y}$, the coefficient of $y'$ in Equations \ref{ex1} or \ref{ex2}, by any rational function $h(y)$ which has no rational antiderivative while drawing the same conclusions. By Corollary \ref{orthcor} and  Claim \ref{theclaim1}, one can replace $ag(y)$ in Equation \ref{ex2} by $g(a,y)$ where $g(s,y)$ is a $k$-algebraic family of rational functions indexed by $\m A^p$ and $a \in \m C^p$ is a point such that for some $k$-specialization $a_0$ of $a$, $g(a_0,y)=0$. 
\end{rem}

%%%%%%%%%%%%%%%%%%%%%%%%%%%%%%%%%%%%%%%%%%%%%%%%%%%%%%%%%%%%%%%%%%%%%%%%%%%%%%%%%%%%%%%%%%%%
\section{Algebraic relations between solutions and orthogonality in the strongly minimal case} \label{formssection}

Let $x_1, \ldots , x_n $ be solutions of Equation (\ref{stareqn}). Since Equation (\ref{stareqn}) is strongly minimal by Theorem \ref{stminthm} and has constant coefficients, by Fact 5.7 and Proposition 5.8 of \cite{casale2020ax}, if $x_1, \ldots , x_n$ are not independent over some differential field $k$ extending $\m C$, then there is a {\em differential} polynomial in two variables (of order zero or one) with coefficients in $\m C$ such that $p(x_i, x_j,x_j')= 0.$\footnote{Note here we are already using strong minimality and triviality to deduce that the relation witnessing non-independence involves the derivative of only one of the solutions.} In this section, we go farther, showing that in our case $p$ can be taken to be a \emph{polynomial} relation between $x_i$ and $x_j$ not involving any derivative. Then in the following section, we give a precise characterization of what the possible polynomial relations between solutions are in terms of basic invariants of the rational function appearing in Equation (\ref{stareqn}) (e.g. singularities, residues).
 
\subsection{Differential forms} \label{5.1}
 
We give some background on differential forms as this will be used heavily in this section.  A general reference on the subject is \cite[Chapter 5]{lang1999fundamentals} in the context of real differential geometry. Recall that throughout, $\mathcal{U}$ is a saturated model of $DCF_0$ with constants $\m C$.
 
 Let $V$ be an irreducible (affine) variety over $\m C$ and let $F=\m C(V)$ be its function field. We identify $\text{Der}(F/\m C)$ with the vector space of rational vector fields of $V(\m C)$, that is a derivation $D\in \text{Der}(F/\m C)$ corresponds to a rational map $$V(\m C)\xrightarrow{X_D} TV(\m C).$$

We let $\Omega^1_V=\Omega^1(F/\m C)$ be the space of rational differential 1-forms on $V(\m C)$ endowed with the universal derivation
$$d: F \rightarrow \Omega^1_V.$$ 
For every derivation $D \in Der(F/\mathbb{C})$,  there exists a unique linear map $D^\ast: \Omega^1_V \rightarrow F$ such that $D^\ast \circ d = D$.  In particular,  the $F$-vector spaces $\text{Der}(F/\m C)$ and $\Omega^1_V$ are dual to each other. It is well known (see \cite[Chapter 2,  Section 8]{Hartshorne}) that any transcendence basis $\xi_1,\ldots,\xi_r$ of $F$ over $\mathbb{C}$ gives rise to a $F$-basis $d\xi_1,\ldots, d\xi_r$ of $\Omega^1_V$
 so that $$dim(V) = ldim_F(\Omega^1_F).$$
In particular,  notice that if  $v=(v_1,\ldots,v_n)$ is a generic point of $V(\mathcal{U})$  then $F=\m C(v)$ and $\{dv_1,\ldots,dv_n\}$ includes a basis for $\Omega_V$.

 For each $n\in\m N$ we define $\Omega^n_V$, the space of rational differential n-forms, to be the exterior algebra $\bigwedge^n\Omega^1_V$. It is the $F$-vector space of all alternating $n$-multilinear maps $$\omega: \text{Der}(F/\m C)^n \rightarrow F.$$  
As usual,  $\Omega^n_V = \{0\}$ for $n > dim(V)$ and otherwise $ldim_F(\Omega^n_V) = {{dim(V)} \choose {n}}$.  In particular, $\Omega^{dim(V)}_V$ is an $F$-vector space of dimension one and an element $\omega \in \Omega^{dim(V)}_V$ will be called a (rational) \emph{volume form} on $V$. 

The finite dimensional $F$-vector space 
$$\Omega^{\bullet}_V = F \oplus \Omega^1_V \oplus \ldots \oplus \Omega^{dim(V)}_V$$
is endowed with the structure of an anticommutative graded $F$-algebra given by the wedge product characterized by the two properties: 
\begin{itemize}
\item[(i)] $\wedge$ is $F$-bilinear. 

\item[(ii)] for every $1$-forms $\omega_1,\ldots, \omega_k \in \Omega^1_V$,
$$ (\omega_1 \wedge \ldots \wedge \omega_k):(D_1,\ldots, D_k) \mapsto det((\omega_i(D_j)_{i,j \leq k}) $$
\end{itemize}

On top of that,  the universal derivative $d: F \rightarrow \Omega^1_V$ extends uniquely into a complex (that is $d \circ d = 0$) of $F$-vector spaces: 
$$0 \rightarrow F \xrightarrow{d} \Omega^1_V \xrightarrow{d}  \Omega^2_V \xrightarrow{d} \ldots \xrightarrow{d} \Omega^n_V \rightarrow 0$$
characterized by the following compatibility condition with $\wedge$:  for every $p$-form $\omega_1$ and $q$-form  $\omega_2$
$$ d(\omega_1 \wedge \omega_2) = \omega_1 \wedge d\omega_2 + (-1)^p \omega_1 \wedge d\omega_2.$$
We refer to \cite{lang1999fundamentals} for more details on the construction outlined above.  

\begin{defn}
Given a derivation $D \in Der(F/\mathbb{C})$, we describe two operations on $\Omega^\bullet_V$ naturally attached to $D$ initially considered by E. Cartan:
\begin{itemize}
\item[(1)] the \emph{interior product} $i_D: \Omega^n_V \rightarrow \Omega^{n-1}_V$ is the contraction  by the derivation $D$: 
 $$i_D\omega(D_1,\ldots,D_{n-1})=\omega(D,D_1,\ldots,D_{n-1}).$$
\item[(2)] The \emph{Lie derivative} $L_D: \Omega^n_V \rightarrow \Omega^n_V$ is defined using ``Cartan's magic formula''
   $$L_D=i_D\circ d+d\circ i_D.$$
\end{itemize}
\end{defn} 

Notice that one can use a different approach to define the Lie derivative based on the Lie bracket of vector fields as described in the first section of \cite{jaoui2020foliations}.  Moreover,  the Lie derivative $L_D$ corresponds to the derivation $D$ defined on $\Omega^\bullet_V$ by Brestovski on page 12 of \cite{brestovski1982deviation}. 

\begin{fact} For $f\in F$, $\omega,\omega_1,\omega_2\in\Omega^n_V$ and $D\in \text{Der}(F/\m C)$, we have the following well-known identities:
\begin{eqnarray*}
L_D(f\omega) &=& D(f)\omega+fL_D(\omega), \\
L_D(\omega_1 \wedge \omega_2) &=& L_D(\omega_1) \wedge \omega_2 + \omega_1 \wedge L_D(\omega_2) \\
L_D (d \omega)&=& d L_D(\omega) , \\ 
L_{fD} &=& f L_D (\omega) + df \wedge i_D (\omega),\\
i_D(\omega_1 \wedge \omega_2) &=& i_D(\omega_1) \wedge \omega_2 + (-1)^n \omega_1 \wedge i_D(\omega_2)
%i_D(\upsilon_1\wedge\upsilon_2)&=&i_D(\upsilon_1)\wedge\upsilon_2-\upsilon_1\wedge i_D(\upsilon_2).
\end{eqnarray*}
\end{fact} 
%{\color{red} We use the last equation for 2-forms in 5.10.  Is it true in general? Do we use the fourth fact?}{\color{blue}: the general formula is 
%$$i_D(\omega_1 \wedge \omega_2) = i_D(\omega_1) \wedge \omega_2 + (-1)^n \omega_1 \wedge i_D(\omega_2) $$
%and we apply it for two $1$-forms (which give a 2-form). Also both formulas 2 and 5 are valid for forms of different degrees but I don't think we use it. I don't think we use the fourth formula anywhere.}

%{\bf JF: I believe I wrote this fact, and my idea here was not to make the presentation minimal, but to include any identities which I think might be useful for people to understand the setting and tools - mainly people from model theory who are learning some of this formalism.}
See \cite[Proposition 5.3 pp. 142]{lang1999fundamentals} for a proof of these identities.  The main definition of this section is
\begin{defn}
Let $(E): y^{(n)} = f(y,y',\ldots, y^{(n-1)})$ be a complex autonomous equation of order $n$ where $f$ is a rational function of $n$ variables.  If $V = \mathbb{C}^n$ with coordinates $x_0,\ldots x_{n-1}$,  the equation $(E)$ defines a derivation $D_f \in Der(\mathbb{C}(V)/\mathbb{C})$  given by: 
$$D_f(x_i) = x_{i +1} \text{ for } i < n - 1 \text{ and } D_f(x_{n-1}) = f(x_0,\ldots,x_{n-1}).$$
We say that a volume form $\omega \in \Omega^{n}_V$ is an \emph{invariant volume form} for the equation $(E)$ if $$L_{D_f}(\omega) = 0.$$
\end{defn}

Before going in further details, we first give an analytic interpretation explaining the terminology although it will not be needed in our analysis.  Consider $$(E): y^{(n)} = f(y,y',\ldots, y^{(n)})$$ a differential equation as above and $\omega$ a (rational) volume form on  $V = \mathbb{C}^n$.   Denote by $U$ the (dense) open set of $V$ obtained by throwing away the poles $f$ and the poles of $\omega$.  As described above, the derivation $D_f$ gives rise to a vector field $s_f$ on $U$  namely the section $s_f: U \rightarrow T(U) \simeq U \times \mathbb{C}^n$ given by
 $$s_f(x_0,\ldots,x_{n-1}) = (x_0,\ldots,x_{n-1};x_1,\ldots,x_{n-1}, f(x_0,\ldots,x_{n-1})).$$
By definition, every point in $\overline{a} \in U$ is a non singular point of the vector field $s_f$.  The classical analytic theorem of local existence and uniqueness for the integral curves of a vector field implies that there exists an analytic function: 

$$ \phi: U_{\overline{a}} \times \mathbb{D} \subset U \times \mathbb{C} \rightarrow U.$$
where $U_{\overline{a}}$ is an analytic neighborhood of $\overline{a}$ and $\mathbb{D}$ a complex disk such that for every $\overline{b} \in U_{\overline{a}}$,  the function $t \mapsto \phi(\overline{b}, t)$ is the local analytic solution of the initial value problem 
$$
\frac{d\phi}{dt} = (\pi_2 \circ s_f)(\phi(t)) \text{ and } \phi(0) = \overline{b}.
$$

We will call $\phi$ the local flow of the vector field $v$ around $\overline{a}$.  The germ of $\phi$ at $(a,0)$ is determined by the vector field $s_f$.

\begin{fact}
With the notation above,  the volume form $\omega$ is invariant for the equation $(E)$ if and only if for every $\overline{a} \in U$,  the local flow $\phi$ around $\overline{a}$ preserves the volume form $\omega$: namely 
$$\text{ for every } t \in \mathbb{D}, \phi_t^\ast \omega_{\phi_t(\overline{a})} = \omega_{\overline{a}}$$
where $\phi_t: U_{\overline{a}} \rightarrow U$ is the function defined by $\phi_t(\overline{b}) = \phi(\overline{b},t)$ and $\omega_p$ denotes the germ of $\omega$ around $p$.
\end{fact}
This follows from the formula on pp. 140 of \cite{lang1999fundamentals}: for every $p$-form $\omega$,
$$\mathcal L_v(\omega) = \frac{d}{dt}_{\mid t = 0} \phi_t^\ast \omega.$$
using the same proof as the proof of \cite[Proposition 3.2.1]{jaoui2020foliations}.  We don't give more details here since this analytic interpretation will not be needed in the rest of the paper. 

Instead, in the following subsections, we will use invariant volume forms together with the following result which follows from the work of Ax \cite{ax1971schanuel}  that can be found explicitly in \cite[Proposition 4]{Rosenlitch2} or \cite[Lemma 6.10]{Marker96modeltheory}.

\begin{fact}\label{AxLemma}
Let $V$ be an irreducible affine variety over $\m C$ and let $F=\m C(V)$ be its function field. Let $u_1,\ldots,u_n,v\in F$ be such that all the $u_i$'s are non zero. Suppose $c_1,\ldots,c_n\in\mathbb{C^*}$ are linearly independent over $\mathbb{Q}$ and let 
$$\omega=dv+\sum_{i=1}^nc_i\frac{du_i}{u_i}.$$
Then $\omega=0$ in $\Omega_V$ if and only if $du_1=\ldots=du_n=dv=0$, i.e $u_1,\ldots,u_n,v\in\mathbb{C}$.
\end{fact}

\subsection{A warm-up case} \label{5.2}

The techniques we will use in our setting already have (known) strong consequences for order one differential equations, where the arguments are often simpler. For instance, the method we use allows us to give a proof a result of Hrushovski and Itai \cite[2.22]{HrIt}.

\begin{lem} \label{twovolumeforms}
Let $V$ be an irreducible (affine) algebraic variety of dimension $n$ and let $D \in Der(\mathbb{C}(V)/\mathbb{C})$ be a derivation.  Assume that the constant field $\mathbb{C}(V)^D$ of $(\mathbb{C}(V),D)$ is equal to $\mathbb{C}$. The space of invariant volume forms 
$$ \Omega^{n}_{V,D} = \lbrace \omega \in \Omega^n_V \mid L_D(\omega) = 0\rbrace$$
is a complex vector space of dimension $\leq 1$.
\end{lem}

\begin{proof}
Clearly,  $\Omega^{n}_{V,D}$ is a complex vector space.  It remains to show that any two non-zero invariant volumes forms $\omega_1,\omega_2 \in \Omega^{n}_{V,D}$ are  linearly dependent.  Since $\Omega^n_V$ is a $\mathbb{C}(V)$ vector space of dimension one,  there exists $f \in \mathbb{C}(V)$ such that $\omega_1 = f \omega_2$.  Computing $L_D$ on both side,  we get: 
$$0 = L_D(\omega_1) = L_D(f\omega_2) = D(f)\omega_2 + f L_D(\omega_2) = D(f) \omega_2.$$ 
Since $\omega_2 \neq 0$,  we get $D(f) = 0$ which implies $f \in \mathbb{C}$.
\end{proof}

In general, this vector space may very well be the trivial vector space but when $V = \mathbb{P}^1$ (or more generally when $V$ is a curve), an easy computation shows that Hrushovski-Itai $1$-form is always an invariant volume form, so that this vector space is always one-dimensional.

\begin{lem} Consider two differential equations of order one of the form: 
$$(E_1): x' = f(x) \text{ and } (E_2):y' = g(y).$$
and denote by $c_1,\ldots,c_r$ the residues of $1/f(x)$ and by $d_1,\ldots,d_s$ the residues of $1/g(x)$.  We assume that $1/f(x)$ and $1/g(y)$ have at least one non zero residue and that $c_1,\ldots,c_r$ are $\mathbb{Q}$-linearly disjoint from $d_1,\ldots,d_s$. That is: 
$$ldim_\mathbb{Q}(c_1,\ldots,c_r) + ldim_\mathbb{Q}(d_1,\ldots,d_s) = ldim_\mathbb{Q}(c_1,\ldots,c_r,d_1,\ldots,d_s).$$
Then $(E_1)$ and $(E_2)$ are weakly orthogonal.
\end{lem}

\begin{proof}
First notice that both the equations $(E_1)$ and $(E_2)$ admit an invariant volume form which are respectively the 1-forms 
$$\omega_1 = \frac {dx}{f(x)} \text{ and } \omega_2 = \frac {dy}{g(y)} $$
associated by Hrushovski and Itai to the equations $(E_1)$ and $(E_2)$. By Lemma \ref{twovolumeforms},  every invariant volume form will be a constant multiple of these forms. So $\omega_1$ and $\omega_2$ are the unique invariant volume forms of $(E_1)$ and $(E_2)$ normalized by 
$$\omega_i(s_i) = 1 \text{ for }i = 1,2.$$
where $s_1(x) = f(x) \frac d {dx}$ and $s_2(y) = g(y) \frac d {dy}$ are the vector fields associated with the derivation $D_f$ on $\mathbb{C}(x)$ and $D_g$ on $\mathbb{C}(y)$ respectively.

For the sake of a contradiction, assume that these two equations are not weakly orthogonal: this means that there exists a closed generically finite to finite correspondence $Z \subset \mathbb{P}^1 \times \mathbb{P}^1$ which is invariant under the derivation $D_f \times D_g$ associated with the product vector field $s_1(x) \times s_2(y)$ on $\mathbb{P}^1 \times \mathbb{P}^1$. Without loss of generality,  we can assume that $Z$ is irreducible.  

Consider the two pull-backs of the two $1$-forms $\omega_1$ and $\omega_2$ (by the respective projections) to $\mathbb{P}^1 \times \mathbb{P}^1$ which are still given by the formulas above in the coordinates $(x,y)$ on $\mathbb{P}^1 \times \mathbb{P}^1$.   Since
$$\mathcal L_{D_f \times D_g}(\omega_1) = \mathcal L_{D_f}(\omega_1) = 0$$
and similarly for $\omega_2$, both $\omega_1$ and $\omega_2$ are invariant $1$-forms for the derivation $D_f \times D_g$ on $\mathbb{P}^1 \times \mathbb{P}^1$.  It follows that their restrictions $\omega_{1 \mid Z}$ and $\omega_{2 \mid Z}$ are two invariant volume forms on $Z$ endowed with the derivation induced by $D_f \times D_g$ on $\mathbb{C}(Z)$.  By Lemma \ref{twovolumeforms}, we conclude that for some $c \in \mathbb{C}$, 
$$(\omega - c \omega_2)_{\mid Z} = 0.$$
Noting the normalization in our case, we see 
$$1 = \omega_1(s_1 \times s_2) = c\omega_2(s_1 \times s_2) = c$$
so that in fact $c = 1$ and the one-form $\omega_1 - \omega_2$ vanishes identically on $Z$.  Write 

\begin{eqnarray*}
\frac{1}{f(x)} = \frac{df_1}{dx} + \sum \frac{c_i}{x - a_i} \\
\frac{1}{g(y)} = \frac{dg_1}{dy} + \sum \frac{d_j}{y - b_j}. 
\end{eqnarray*} 

Using this notation, we have an equality of $1$-forms on $Z$ 
\begin{eqnarray*}
 df - dg  &= \sum c_i\frac{d(x-a_i)}{x - a_i} +  \sum d_j\frac{d(y-b_j)}{y - b_j}  \\
 & = \sum \alpha_i \frac{df_i} {f_i} + \sum \beta_j\frac{dg_j}{g_j}.
\end{eqnarray*}
where the $\alpha_i$ forms a $\mathbb{Q}$-basis of $c_1,\ldots,c_n$, the $\beta_j$ form a $\mathbb{Q}$-basis of $ d_1,\ldots, d_s$ and $f_i(x) \in \mathbb{C}(x), g_j(y) \in \mathbb{C}(y)$. Note that a linear combination of logarithmic derivatives can always be rewritten as a sum of logarithmic derivatives in which the coefficients are linearly independent over $\m Q$. See Remark on page 76 of \cite{Marker96modeltheory}. 

The assumption of linear independence means that the $\beta_j$ and the $\alpha_i$ form a $\mathbb{Q}$-linearly independent set.  By Fact \ref{AxLemma}, 
$f_i(x)$ is constant on $Z$ for all $i$ and $g_j(y)$ is constant on $Z$ for all $j$. Since $f(x)$ and $g(y)$ have at least one non-zero residue, we conclude that $Z$ can not project dominantly on the solution sets of $(E_1)$ and $(E_2)$. Contradiction.
\end{proof}

\subsection{Our setting} \label{5.3}

Let $f\in\m C(z)$ be a rational function and consider the associated equation $(\star)$. Let $V=\m C^2\setminus Z_f$ in coordinates $(x,y)$, where $Z_f$ is the union of horizontal line $y = 0$ and, for each pole $a$ of $f$, the vertical line $x = a$. Consider the section of the tangent bundle $s_f:V\rightarrow T(V)$
  $$s_f(x,y)=(x,y,y,yf(x)).$$ Let $\pi_2 s_f (x,y) := (y,yf(x)).$Then we showed, in Section 2, that if $f(z)$ is such that for any $h\in \m C(z)$ $f(z)\neq\frac{d h}{dz}$, it follows that $(V,s_f)^\#=\{(x,y)\in V_f(\mathcal{U}): x^\prime=y\land y^\prime=yf(x)\}$ is a geometrically trivial strongly minimal set.
  
 This section $s_f$ gives rise to the derivation $D_{f}\in \text{Der}(\m C(V)/\m C)$ given by
  $$D_{f}(h)=y \frac{\partial h} {\partial x}+{yf(x)}\frac{\partial h} {\partial y}.$$
  In particular, $D_{f}(x)=y$, $D_{f}(y)=yf(x)$ and $D_{f}(\frac{1} {y})={\frac{-f(x)} {y}}.$

\begin{lem}\label{Volume form}
For any $f\in\m C(z)$, the derivation (or vector field) $D_{f}$ preserves the volume form
$$\omega={{dx\wedge dy}\over y}\in \Omega^{2}_V$$ and $L_{D_f}(\omega)=0$.
\end{lem}
\begin{proof}
We only need to show that $L_{D_f}(\omega)=0$.
\begin{eqnarray}
L_{D_f}(\omega)&=&L_{D_f}({{dx\wedge dy}\over y})\notag\\
&=&D_f({1\over y})({dx\wedge dy})+{1\over y}L_{D_f}({dx\wedge dy})\notag\\
&=&{-f(x)\over y}({dx\wedge dy})+{1\over y}\left[{dy\wedge dy}
+dx\wedge(f'(x)ydx+f(x)dy)\right]\notag\\
&=&{-f(x)\over y}({dx\wedge dy})+{1\over y}\left[f'(x)ydx\wedge dx+f(x)dx\wedge dy)\right]\notag\\
&=&0\notag
\end{eqnarray}
\end{proof}

Now assume that $f(z),g(z)\in \m C(z)$ are such that for any $h\in \m C(z)$, we have that neither $f(z)\neq\frac{d h}{dz}$ nor $g(z)\neq\frac{d h}{dz}$. {We do not exclude here the possibility that $f(z)=g(z)$}. Let $V=\m C^2\setminus Z_f$ and $W=\m C^2\setminus Z_g$ with coordinates $(x_1,y_1)$ and $(x_2,y_2)$ respectively. Assume that the two strongly minimal definable sets $(V,s_f)^\#$ and $(W,s_g)^\#$ are nonorthogonal. Then since they are geometrically trivial, they are non-weakly orthogonal. So there is $Z\subset V\times W$ a closed complex  $D_f\times D_g$ invariant generically finite to finite correspondence witnessing nonorthogonality. We write 
$$\omega_1={{dx_1\wedge dy_1}\over y_1}\in \Omega^{2}_V\;\;\text{and}\;\;\omega_2={{dx_2\wedge dy_2}\over y_2}\in \Omega^{2}_W$$ for the corresponding $2$-forms. From Lemma \ref{Volume form}, we have that $L_{D_f}(\omega_1)=L_{D_g}(\omega_2)=0$.
We will now view $\omega_1$ and $\omega_2$ as $2$-forms on $Z$, which are volume forms since $Z$ is a finite to finite correspondence ($\text{tr.deg.}_{\m C}\m C(Z)=2$). More precisely we let $\tilde{\omega}_1$ be the 2-form on $Z$ defined as the pullback of $\omega_1$ by the projection map $\pi_1:Z\rightarrow V$. The form $\tilde{\omega}_2$ is defined similarly. By construction we have that
$$L_{D_f\times D_G}(\tilde{\omega}_1)=L_{D_f}(\omega_1)=0.$$ A similar expression holds for $\tilde{\omega}_2$.

\begin{lem}\label{innerproduct}
Let $Z$ be as above, then there exist $c\in\m C^*$ such that 
$$i_{D_f\times D_g}\left({{dx_1\wedge dy_1}\over y_1}-c\cdot {{dx_2\wedge dy_2}\over y_2} \right)=0$$
where $i_{D_f\times D_g}$ is the interior product.
\end{lem}
\begin{proof}
Since $Z$ is $2$-dimensional, the space of rational $2$-forms on $Z$ is a $\m C(Z)$-vector space of dimension
one. So there exists $h\in \m C(Z)$ such that
$$\tilde{\omega}_1=h\tilde{\omega}_2.$$
We hence have that 
\begin{eqnarray}
0&=&L_{D_f\times D_G}(\tilde{\omega}_1)\notag\\
&=&L_{D_f\times D_G}(h\tilde{\omega}_2)\notag\\
&=&(D_f\times D_g)(h)\omega_2+hL_{D_g}(\omega_2)\notag\\
&=&(D_f\times D_g)(h)\omega_2\notag
\end{eqnarray}
So $h$ is in the constant field of $D_f\times D_g$ in $\m C(Z)$. Since the equations are orthogonal to $\m C$, we have that $h\in\m C$. By construction we have that ${\omega}_1=c{\omega}_2$, for $c\in\mathbb{C}$.

Hence on $Z$, the two form ${\omega}_1-c{\omega}_2={{dx_1\wedge dy_1}\over y_1}-c\cdot {{dx_2\wedge dy_2}\over y_2}$ is identically $0$. Furthermore, on $Z$, the $1$-form obtained by applying the interior product $i_{D_f\times D_g}$ is 0 and the result follows.
\end{proof}

\begin{lem} \label{constantlem}
Let $Z$ be as above, then there is $c\in\m C^*$ such that on $Z$
$$dy_1-f(x_1)dx_1-c(dy_2-f(x_2)dx_2)=0\in\Omega_Z^1$$
\end{lem}
\begin{proof}
We will use the formula
$$i_D(\upsilon_1\wedge\upsilon_2)=i_D(\upsilon_1)\wedge\upsilon_2-\upsilon_1\wedge i_D(\upsilon_2)$$where $D$ is any derivation and $\upsilon_1,\upsilon_2$ are 1 forms. Starting with lemma \ref{innerproduct}
\begin{eqnarray}
0&=&i_{D_f\times D_g}\left({{dx_1\wedge dy_1}\over y_1}-c\cdot {{dx_2\wedge dy_2}\over y_2} \right)\notag\\
&=&i_{D_f}({{dx_1\wedge dy_1}\over y_1})-c\cdot i_{D_g}({{dx_2\wedge dy_2}\over y_2})\notag\\
&=&{i_{D_f}(dx_1)\wedge dy_1-dx_1\wedge i_{D_f}(dy_1)\over y_1}-c\cdot{i_{D_f}(dx_2)\wedge dy_2-dx_2\wedge i_{D_f}(dy_2)\over y_2}\notag\\
&=&{y_1dy_1-y_1f(x_1)dx_1\over y_1}-c\cdot{y_1dy_2-y_2f(x_2)dx_2\over y_2}\notag\\
&=& dy_1-f(x_1)dx_1-c(dy_2-g(x_2)dx_2)\notag
\end{eqnarray}
\end{proof}
\begin{prop}
Let $Z$ be as above. Then $Z$ is contained in a closed hypersurface of $V\times W$ of the from $Z(p)$ for some $p\in\m C[x,y]$.
\end{prop}
\begin{proof}
Recall that by assumption for some $f_1,g_1\in\m C(z)$, we have that
$$f(x_1)=\frac{d f_1}{dx_1}+\sum{\frac{c_i}{x_1-a_i}}$$
and
$$g(x_2)=\frac{d g_1}{dx_2}+\sum{\frac{d_i}{x_2-b_i}}$$
where at least one of the $c_i$'s and one of the $d_i$'s are non-zero. Multiplying the above equations by $dx_1$ and $dx_2$ respectively and using 
$$dy_1-f(x_1)dx_1-c(dy_2-g(x_2)dx_2)=0$$ 
we get
$$d(y_1-cy_2-f_1(x_1)+cg_1(x_2))=-\sum{c_i\cdot\frac{d(x_1-a_i)}{x_1-a_i}}+\sum{cd_i\cdot\frac{d(x_2-b_i)}{x_2-b_i}}.$$
We use here that $d(x_1-a_i)=dx_1$ and $d(x_2-a_i)=dx_2$. Consider the $\m Q$-linear span of $\{c_i,cd_j\}$ - which is a non-trivial vector space since $f(z)$ (and on top of that $g(z)$) has at least one simple pole - and extract $\{e_1,\ldots,e_s\}$ a $\m Q$-basis (so $s\geq 1$). 

If we divide all the $e_i$'s by some $N\gg0$, we can assume that $c_i$'s and $cd_j$'s are in the $\m Z$-span and get that
$$d(y_1-cy_2-f_1(x_1)+cg_1(x_2))=\sum{e_k{dh_k\over h_k}}$$
where $h_k\in\m C[x_1,x_2]$ has the specific form 
$$h_k =  \prod (x_1 - a_i)^{-n(c_i,k)} \prod (x_2-b_j)^{n(cb_j,k)}$$
and $n(c_i,k)$(resp.  $n(cb_j,k)$)  denotes the coefficient of $c_i$ (resp.  $cb_j$) relatively to $e_k$ in the basis $e_1,\ldots,e_k$. But by Fact \ref{AxLemma}, it must be that
$$d(y_1-cy_2-f_1(x_1)+cg_1(x_2))=0\;\;\text{and}\;\;{{dh_k}}=0.$$ Hence for $k=1$ as an example, we get that $$h_1(x_1,x_2)=c$$ for some constant $c\in \m C$. Since $Z$ projects dominantly on $V$ and $W$, we get a non-trivial polynomial relation between $x_1$ and $x_2$ as required.

\end{proof} 

To summarize, in this subsection, we have shown
\begin{prop}\label{StrongTriviality}
Let $f(z),g(z)\in \m C(z)$ be such that for any $h\in \m C(z)$, we have that neither $f(z)\neq\frac{d h}{dz}$ nor $g(z)\neq\frac{d h}{dz}$. Suppose that $x$ and $y$ are solutions to the strongly minimal equations $$\frac{z''}{z'} = f(z)\;\;\;\text{and}\;\;\;\frac{z''}{z'} = g(z)$$ respectively. Let $K$ be any differential extension of $\m C$ such that $K\gen{y}^{alg}=K\gen{x}^{alg}$. Then $\m C(x)^{alg}=\m C(y)^{alg}$.
\end{prop}

In the next section we classify the algebraic relations between solutions in details in the case that $\m C(x)^{alg}=\m C(y)^{alg}$, in particular showing that there are only finitely many, depending on basic invariants of the rational functions $f,g$.

%%%%%%%%%%%%%%%%%%%%%%%%%%%%%%%%%%%%%%%%%
\section{Algebraic relations between solutions} \label{algrelsection}

\begin{thm}\label{alg-solutions-with-derivatives}
Let $f_1(z),\ldots, f_n(z) \in \mathbb{C}(z)$ be rational functions such that each $f_i(z)$ is not the derivative of a rational function in $\mathbb{C}(z)$ and consider for $i = 1,\ldots, n$, $y_i$ 	a solution of 
$$(E_i): y''/y' = f_i(y)$$
Then $trdeg_\mathbb{C}(y_1,y'_1,\ldots, y'_n,y_n) = 2n$ unless for some $i \neq j$ and some $(a,b) \in \mathbb{C}^\ast \times \mathbb{C}$, $y_i = ay_j + b$. In that case,  we also have $f_i(z) = f_j(az + b)$.
\end{thm}
Notice that we do not exclude the case where some of the $f_i(z)$ are equal in this statement. 

\begin{proof}
By Theorem \ref{stminthm} and Corollary \ref{triviality}, we already know that each of the equations
$$(E_i): y''/y' = f_i(y)$$ 
is strongly minimal and geometrically trivial.  It follows that if $y_1,\ldots, y_n$ are solutions of $(E_1),\ldots,(E_n)$ such that  $trdeg_\mathbb{C}(y_1,y'_1,\ldots, y'_n,y_n) = 2n$ then for some $i \neq j$, 
$$trdeg_\mathbb{C}(y_i,y'_i,y_j,y'_j) < 4$$
Since all the equations $(E_i)$ do not admit any constant solution,  $y_i$ and $y_j$ must realize the generic type of $(E_i)$ and $(E_j)$ respectively and using strong minimality we can conclude that 
$$trdeg_\mathbb{C}(y_i,y'_i,y_j,y'_j) = 2 \text{ and } \mathbb{C}(y_i,y_i')^{alg} = \mathbb{C}(y_j,y_j')^{alg}.$$
Proposition \ref{StrongTriviality} now implies that in fact $\mathbb{C}(y_i)^{alg} = \mathbb{C}(y_j)^{alg}$.  To simplify the notation,  set $f(z) = f_i(z)$, $g(z) = f_j(z)$, $y_i = x$ and $y_j = y$ and so we have that $x$ and $y$ are interalgebraic over $\m C$.  

First note that the derivation on ${\m C}(x)^{\rm alg}$ has image in the module ${\m C}(x)^{\rm alg}x^\prime$.  If $F(x,z)=0$ for some $z$,
then we have that $z^\prime=-{F_x(x,z)\over F_z(x,z)}x^\prime$.  Thus there are $\alpha,\beta\in {\m C}(x)^{\rm alg}$ such that $y^\prime=\alpha x^\prime$ and $\alpha^\prime=\beta x^\prime$.  Then $$y^{\prime\prime}=\beta (x^\prime)^2+\alpha (x^{\prime\prime})=\beta (x^\prime)^2+\alpha f(x)x^\prime$$
but also
$$y^{\prime\prime}=g(y)y^\prime=\alpha g(y)x^\prime.$$
Since $x^\prime\ne 0$, $$\beta x^\prime=\alpha(g(y)-f(x)).$$
If $\beta\ne 0$, then $x'=\frac{\alpha(g(y)-f(x))}{\beta}\in{\m C}(x)^{\rm alg}$ contradicting strong minimality.
Hence $\beta=0$. Since $\alpha^\prime=\beta x^\prime$ and $y^\prime$(=$\alpha x^\prime$) is not zero, we get that $\alpha\in {\m C}^\times$. Using $y^\prime=\alpha x^\prime$ we also obtain that $y=\alpha x+ b$ for some $b\in\mathbb{C}$. Finally, $\beta=0$ also implies that $f(x)-g(y)=0$ and hence $f(x)=g(y)=g(\alpha x+b)$. 
%We can write $y=R(x)$ for some algebraic function $R$, with $R_x\neq 0$. Then by direct computations and applying the equations:
%\begin{eqnarray*}
%y &=& R \\
%y' &=& R_x x' \\
%y''&=& R_{xx}(x')^2 +R_x x'' \\
%y'g(y)&=& R_{xx} (x')^2 + R_xx'f(x) \\
%R_x x'g(R)&=& R_{xx} (x')^2 + R_xx'f(x) \\
%\end{eqnarray*} 
%Rearranging terms, we obtain: 
%\begin{eqnarray*}
%x' R_{xx} = R_x(f(R)-f(x)).
%\end{eqnarray*} 
%This equation is a nontrivial order one differential polynomial as long as $x' R_{xx}\neq 0$, and it vanishes at $x$. But, this is impossible by strong minimality. So, we must have $x' R_{xx}=0,$ and hence it must be that $R_{xx}=0$ and $R_x(g(R)-f(x))=0$. Since $R_x\neq 0$, the first equality gives that $R(x)=ax+b$ for $a\in \mathbb C^*$ and $b\in \mathbb C$ so that $y = ax + b$, and the second equality gives that $g(ax+b) =f(x)$. 
%Suppose $x$ and $y$ are transcendental over $\m C$, $x^{\prime\prime}=f(x)x^\prime$, $y^{\prime\prime}=g(y)y^\prime$ and
\end{proof}

In the rest of this section, we will derive some consequences of Theorem \ref{alg-solutions-with-derivatives} on the structure of the solution sets of these equations.  
First let us recall the definitions of what it means for an equation to have no or little structure.

\begin{defn} Suppose that $X$ is a geometrically trivial strongly minimal set defined over some differential field $K$. Then, $X$ is said to be {\em $\omega$-categorical} if for any $y\in X$, the set $X\cap K\gen{y}^{alg}$ is finite. Moreover, if $X\cap K\gen{y}^{alg}=\{y\}$, then we say that $X$ is {\em strictly disintegrated}.
\end{defn}

\begin{exam}
In the Poizat example, that is when $f(z)=\frac{1}{z}$, the requirement $\frac{1}{az+b}=\frac{1}{z}$ gives that $a=1$ and $b=0$. Hence, the Poizat example is strictly disintegrated.
\end{exam}

\begin{exam}
Consider now the case where $f(z)=\frac{1}{z-a}-\frac{1}{z-b}$, where $a,b\in\m C$. Then it is not hard to see that $f(-z+a+b)=f(z)$. Hence it follows that the strongly minimal equation $\frac{z''}{z'}=\frac{1}{z-a}-\frac{1}{z-b}$ is not strictly disintegrated. Moreover we will show that it is $\omega$-categorical.
\end{exam}
%{\color{red} We need to add a few words about the case when $f(z)\neq g(z)$:  for example orthogonality follows as soon as the ``number of nonzero residues'' are not equal.}

We now focus on the case when $f(z)=g(z)$ and will further study the condition $f(ax+b) =f(x)$. Recall that $f(z)$ is such that $\frac{z''}{z'} = f(z)$ is strongly minimal. We write $f(z) = \frac{dg}{dz} + \sum _{i=1}^n \frac{c_i }{z- \alpha _i}$ and so  $f(ax+b) =f(x)$ gives 
$$\frac{dg}{dz} (\beta (x)) + \sum _{i=1}^n \frac{c_i }{\beta (x)- a _i}  = f(x) $$ where $\beta(x)=ax+b$. Since $\beta$ has such a simple form, it is easy to see that $\beta$ must permute the set of $a_i,$ points at which $f$ has a nonzero residue or else $f(ax+b)\neq f(x)$. So, bounding the size of the setwise stabilizer of the collection of $a_i$ will bound the number of nontrivial algebraic relations between solutions. In what follows, we let $A$ be the collection of $a_i$ at which $f(x)$ has a nontrivial residue and $G_1$ be the stabilizer of $A.$ We assume that the points of $A$ have a unique orbit under $G_1$ - otherwise replace $A$ by one of the orbits. Our arguments below will only depend on the size of any particular set stabilized by the affine transformations which induce algebraic relations. 

For some $n$, $\beta ^n $ is in the \emph{pointwise} stabilizer any of the collection of $a_i \in A$. If there is more than one $a_i,$ then $\beta ^n$ is the identity, since the pointwise stabilizer of two distinct points under the group of affine transformations is trivial (e.g.  directly from the stabilizer condition, one gets two linearly independent equations for $a,b$ and of course $a=1, b=0$ is a solution to the system - thus the unique solution). So, $\beta $ is torsion in the group of affine transformations. We will represent the group of affine transformations in the standard manner: 
$$\left\{ \begin{pmatrix}
a & b \\
0 & 1
\end{pmatrix} \, \middle| \, a,b \in \m C \right\}.$$ The natural action on $x \in \m C$ is given by matrix multiplication on the vector $\begin{pmatrix}
x \\
1 
\end{pmatrix}$. One can show that the elements of finite order in this group are precisely those in which $a$ is a root of unity of some order greater than one together with the identity element. When $a$ is a primitive $k^{th}$ root of unity, the cyclic subgroup of the affine generated by $\begin{pmatrix}
a & b \\
0 & 1
\end{pmatrix}$ is of order $k$. If $|A|=1$, then a simple argument shows that there are no nontrivial affine transformations which preserve $f(x)$. In the case that $|A|>1$, there is an upper bound on the number of affine transformations preserving $f$ in terms of $|A|$ (the same argument works with any set known to be stabilized by the action). 
\begin{claim} \label{2transcl}
When $|A|=n,$ $|G|\leq n(n-1).$ 
\end{claim}
\begin{proof}
First, note that the action of the affine group on the affine line is \emph{sharply $2$-transitive}, meaning that for any pairs of distinct elements $(c_1,c_2)$ and $(d_1,d_2)$ in $\m C^2$, there is precisely one affine transformation which maps $(c_1,c_2)$ to $(d_1,d_2).$ Thus, the action of the setwise stabilizer on the collection of $a_i$ will be determined by determining the image of $a_1$ and $a_2.$ Since their images are in the collection $\{a_1, \ldots , a_n \}$ there are at most $n(n-1)$ choices for their images, of which at most $n(n-1)-1$ correspond to nontrivial affine transformations. Thus the setwise stabilizer is of size at most $n(n-1).$ 
\end{proof}

\begin{cor}
Let $f(z) \in \mathbb{C}(z)$ which is not the derivative of a rational function. Then the solution set of the equation $(E): y''/y' = f(y)$ is $\omega$-categorical.  
\end{cor}

\begin{proof}
We already proved that the equation is strongly minimal and geometrically trivial. All we need to show is that if $y$ is a solution of $(E)$, then $\mathbb{C}\gen{y}^{alg}$ only contains finitely many solutions of $(E)$. By Theorem \ref{alg-solutions-with-derivatives} applied to $f_1(y) = f_2(y) = f(y)$,  we see that if $y_1 \in \mathbb{C}\gen{y}^{alg}$ is a solution of $(E)$, then $y_1 = ay + b$ for some $(a,b) \in \mathbb{C}^\ast \times \mathbb{C}$ such that  $z \mapsto az + b$ belongs in the stabilizer of $f(z)$ by the action of $Aff_2(\mathbb{C})$ on $\mathbb{C}(z)$ by precomposition. It follows that if $k$ is the number of solutions of $(E)$ in $\mathbb{C}\gen{y}^{alg}$, then
$$ k = |Stab(f(z))| \leq n(n-1) $$
where $n$ is the number of non-zero complex residues of $f(z)$.
\end{proof}

Actually, the previous bound for $k$ obtained above is not sharp. Before improving the bound, we first give an example for which $k$ will be maximal based on the number of non zero residues of $f(y)$.

\begin{exam}
Let $n \geq 2$, let $c \in \mathbb{C}^\ast$, let $\xi$ be a primitive $n$-th root of unity and let $g(z) \in \mathbb{C}(z)$ be a rational function.  Consider
$$ f(z) = c.\sum_{k = 0}^{n-1} \frac{\xi^k}{z - \xi^k} + g(z^n) \in \mathbb{C}(z).$$
We claim that $f(\xi z) = f(z)$.  Indeed,   obviously $g((\xi z)^n) = g(z^n)$ and moreover 
\begin{eqnarray*}
\sum_{k = 0}^{n-1} \frac{\xi^k}{\xi z - \xi^k} & = &  \sum_{k = 0}^{n-1} \frac{\xi^{k-1}}{z - \xi^{k-1}} \\ 
& = & \frac {\xi^{-1}}{z - \xi^{-1}} +  \sum_{k = 1}^{n-1} \frac{\xi^{k-1}}{z - \xi^{k-1}} \\
& = &  \frac {\xi^{n-1}}{z - \xi^{n-1}} + \sum_{k = 0}^{n-2}   \frac{\xi^{k}}{z - \xi^{k}} \\
& = & \sum_{k = 0}^{n-1}   \frac{\xi^{k}}{z - \xi^{k}}.
\end{eqnarray*}
It follows that $f(\xi^k z ) = f(z)$ for all $k \leq n-1$ and therefore that the stabilizer $f(z)$ under the action of the affine group has cardinal $\geq n$. Consequently,  there at least $n$ polynomial relations for the solutions of the differential equation $\frac{y''}{y'} = f(y)$.
\end{exam}

The following lemma shows that this in fact equality holds:
\begin{lem} \label{stablemma}
Let $f(z) \in \mathbb{C}(z)$ be a function with at least one non zero residue.  Denote by $G$ the stabilizer of $f(z)$ under the action of the affine group by precomposition and by $n \geq 1$ the number of complex points where $f(z)$ has a non zero residue.  Then 
$$|G| \leq n.$$
\end{lem}

We already know by Claim 3.3 that $G$ is finite. 

\begin{claim}
Any finite subgroup $G$ of $\Aff_2(\mathbb{C})$ is cyclic and conjugated to a finite subgroup of rotations (for the usual action of  $\Aff_2(\mathbb{C})$ on the complex plane). 
\end{claim}

\begin{proof}
Since the additive group $\mathbb{G}_a (\mathbb{C})$ has no non-trivial finite subgroup,  using the exact sequence 
$$0 \rightarrow \mathbb{G}_a (\mathbb{C}) \rightarrow \Aff_2(\mathbb{C}) \rightarrow \mathbb{G}_m(\mathbb{C}) \rightarrow 1,$$
we see that $G$ is isomorphic to its image $\mu(G)$ in $\mathbb{G}_m(\mathbb{C})$ and therefore that $G$ is cyclic.  Moreover,  in the matrix representation of $\Aff_2(\mathbb{C})$,  $G$ is generated by an element of the form

$$ \Xi = \begin{pmatrix} \xi & b \\ 0 & 1 \end{pmatrix}$$
where $\xi$ is a root of unity.  A direct computation shows that 
$$\begin{pmatrix} 1 & -c \\ 0 & 1 \end{pmatrix} \begin{pmatrix} \xi & 0 \\ 0 & 1 \end{pmatrix}\begin{pmatrix} 1 & c \\ 0 & 1 \end{pmatrix} = \begin{pmatrix} \xi & (\xi - 1) c \\ 0 & 1 \end{pmatrix}$$
Hence,  if $\xi \neq 1$ (i.e.  $G$ is not the trivial group) then taking $c = \frac b {\xi - 1}$ conjugates $G$ to a subgroup of $\mathbb{S}^1  = \lbrace z \in \mathbb{C} \mid \mid z \mid = 1  \rbrace \subset \mathbb{G}_m(\mathbb{C})$ (i.e.  a subgroup of rotation of the complex plane).  
\end{proof}

\begin{claim}
Let $f(z) \in \mathbb{C}(z)$ be a rational function stabilized by a non-trivial finite group $G$ of rotations of the complex plane then  $f(z)$ has a trivial residue at $z = 0$. 
\end{claim}
\begin{proof}
The finite group $G$ is generated by the rotation $z \mapsto \xi z$ where $\xi \neq 1$ is a root of unity.  Write 
$$f(z) = \frac a z + g(z)$$
where $0$ is not a simple pole of $g(z)$.  Therefore,  $0$ is not a simple pole of $g(\xi z)$ either and 
$$ f(\xi z) = \frac {a \xi^{-1}} z + g(\xi z).$$
Comparing the residues of $f(\xi z)$ and $f(z)$ at $0$, we get $a = a \xi^{-1}$ and therefore $a = 0$.  Hence,  $f(z)$ has a trivial residue at $z = 0$.
\end{proof}

\begin{proof}[Proof of Lemma \ref{stablemma}]
Denote by $G$ the stabilizer of $f(z)$.  We already know that $G$ is finite by Claim 3.3.  Using Claim 3.8, up to replacing $f(z)$ by $f(z + c)$,  we can assume that $G$ is a subgroup of the group of rotations of the complex plane since this transformation does not affect the number of complex points where $f(z)$ has a non-zero residue.  

In particular, $G$ is a subgroup of $\mathbb{G}_m(\mathbb{C})$ acting on the complex plane by multiplication.  Denote by $A = \lbrace a_1,\ldots, a_n \rbrace \neq \emptyset$ the set of complex points where $f(z)$ has a non-zero residue.  The second claim ensures that $A \subset \mathbb{C}^\ast$.  Since the action of $\mathbb{G}_m(\mathbb{C})$ on $\mathbb{C}^\ast$ is $1$-sharply transitive, the same argument as in Claim 3.3 gives
$$|G| \leq n.$$
\end{proof} 

Since $G$ is a group of rotations,  the proof gives in fact a bit more: if the upper bound is achieved ($|G|= n$) then all the complex numbers where $f(z)$ has a non zero residue must lie on a common circle of the complex plane.

Coming back to Example 3.6,  since $n \geq 2$,  $g(z)$ does not have non zero residues and therefore $f(z)$ has a non-zero residue exactly at the $n^{th}$ roots of unity.  It follows that the stabilizer of $f(z)$ is exactly the group of rotations of the complex plane with angles $2\pi k/n$ with $k = 0,\ldots {n-1}$.

\begin{lem}
Let $f(z) \in \mathbb{C}(z)$ be a rational function with at most simple poles.  Assume that $f(z)$ has a non-zero residue at $n \geq 2$ complex points and that equality occurs in the previous lemma: 
\begin{center} 
the stabilizer of $f(z)$ under the action of the affine group by precomposition has cardinality $n$. 
\end{center} 
Then $f(z)$ is conjugated to one of the examples of Example 3.6: there exist $a,b \in \mathbb{C}$ such that
$$f(az + b) = c.\sum_{k = 0}^{n-1} \frac{\xi^k}{z - \xi^k} + g(z^n)$$
where $g(z) \in \mathbb{C}[z]$ is a polynomial.
\end{lem}

\begin{proof}
As for the proof of the previous lemma,  replacing $f(z)$ by $f(z+c)$, we can assume that the stabilizer $G$ of $f(z)$ is the subgroup of rotations with angles $2\pi k/n$ for $k = 0,\ldots, n-1$.  As noticed after the proof of the previous lemma,  after this translation,  all the poles of $f(z)$ lie (in a single orbit hence) on a circle centered at $0$ (say of radius $r$).  Replacing $f(z)$ by $f(rz)$, we can assume that all the poles of $f(z)$ lie on the unit circle.  Finally, replacing $f(z)$ again by $f(e^{i\theta}z)$, we can assume one of the pole of $f(z)$ is $z = 1$.

After this combination of affine substitutions,  the $n$ simple poles of $f(z)$ are located at $n^{th}$ roots of unity  $1,\xi,\ldots, \xi^{n-1}$.  We claim that 
$$f(z) = c.\sum_{k = 0}^{n-1} \frac{\xi^k}{z - \xi^k} + g(z^n).$$
Indeed,  writing the partial fraction decomposition of $f(z)$ as
$$ f(z) = P(z) + \sum_{i = 0}^{n-1} \frac {\alpha_i}{z - \xi ^i}$$
we get (by uniqueness of the partial fraction decomposition) that both terms are preserved under the action of $G$. 
\begin{itemize} 
\item Looking at the polynomial part,  if $a$ is a root of $P$ then 
$$ P(a) = P(\xi a) = \cdots = P(\xi^{n-1} a) = 0.$$
\end{itemize}
Since these are all distinct roots,  we get that 
$$ (z-a)(z-\xi a) \cdots(z-\xi^{n-1} a) = (z^n - a^n)$$
divides $P(z)$. Iterating the argument,  we obtain that $P(z)$ is of the form
$$P(z) = (z^n - a_1^n) \cdots (z^n - a_k^n) = g(z^n)$$

\begin{itemize}
\item Looking at the simple poles part: we compute as in Example 3.6

\begin{eqnarray*}
\sum_{i = 0}^{n-1} \frac {\alpha_i}{\xi z - \xi ^i} & = & \sum_{i = 0}^{n-1} \frac {\alpha_i {\xi}^{-1}}{z - \xi ^{i-1}} \\
& = & \frac {\alpha_0 {\xi}^{-1}}{z - \xi^{n-1}} +  \sum_{i = 0}^{n-2} \frac {\alpha_{i+1} {\xi}^{-1}}{z - \xi ^{i}}
\end{eqnarray*} 
\end{itemize}
which gives $\alpha_{i+1} = \alpha_i \xi$ for $i = 1,\ldots (n-2)$ and $\alpha_{0} = \alpha_{n-1} \xi$.  In particular, $\alpha_0 = c$ can be chosen freely and $\alpha_i = \xi^i c$ for $i \geq 1$. The last equality is automatically satisfied since $\xi$ is a $n^{th}$-root of unity.

Putting everything together,  we showed that after these substitution,  we obtain
$$f(z) = g(z^n) + c. \sum_{k = 0}^{n-1} \frac{\xi^k}{z - \xi^k}. $$
\end{proof}

\begin{exam}
If we consider the functions given in Example 3.2 by 
$$ f(z) = \frac 1 {z- a} - \frac 1 { z- b} $$
and $z \mapsto \frac {a - b}{2}z + \frac {a + b}{2}$ is the unique affine transformation sending $(1,-1)$ to $(a,b)$ then 
$$ f(\frac{a - b}{2}z + \frac {a + b}{2}) = \frac {2} {b - a} (\frac 1 {z + 1} - \frac 1 {z-1})$$
are all of the form prescribed by the lemma.

On the other hand, it is necessary to assume that $f(z)$ has only simple poles for the conclusion of the lemma to hold. For instance, 
$$ f(z) = \frac {-1} {z-1} + \frac 1 {z+1} + \frac 1 {(z-a)^2} + \frac 1 {(z+a)^2} +\frac 1 {(z+ b)^3} - \frac 1 {(z-b)^3}  $$
is not of the form given by Example 3.6 and satisfies $f(z) = f(-z)$.
\end{exam}

\begin{cor}
Let $f(z) \in \mathbb{C}(z)$. Denote by $a_1,\ldots, a_n$ the non-zero zero complex residues of $f(z)$ and assume $n \geq 1$.  For a solution $y$ of $(E): y''/y' = f(y)$, denote by $acl(y)$ the set of solutions of $(E)$ which are algebraic over $\mathbb{C}\gen{y}$.  Then $| acl(y) | $ does not depend on the chosen solution $y$ and 
$$ 1 \leq | acl(y) | \leq n$$
Moreover, 
\begin{itemize}
\item[(i)] $| acl(y) | = 1$ if $n = 1$ or if $n \geq 3$ and the only affine transformation which preserves the set of complex residues of $f(z)$ is the identity.  In that case, the equation is strictly disintegrated. 

\item[(ii)]Assume that $f(z)$ does not have higher order poles.  Then $| acl(y) | = n$ if and only if for some $(a,b) \in \mathbb{C}^\ast \times \mathbb{C}$,
$$f(az + b) = c.\sum_{k = 0}^{n-1} \frac{\xi^k}{z - \xi^k} + g(z^n)$$
where $g(z) \in \mathbb{C}[z]$ is a polynomial and $\xi$ is a primitive $n$-root of unity.
\end{itemize}
\end{cor}

\section{Observations about the non-minimal case} \label{nonmin} 
Theorem \ref{stminthm} tells us that the solution set of $\frac{z''}{z'}=f(z)$ has rank $2$ precisely when we can write $f(z)$ as the derivative of a rational function $g(z)$. In that case, a family of order one subvarieties fibers our equation and is given by $z' = g(z)+c$ for $c \in \m C$. A priori, three options might arise: 
\begin{enumerate} 
\item The equation $\frac{z''}{z'}=f(z)$ is internal to the constants. 
\item The fibers $z' = g(z)+c$ are internal to the constants, but the equation $\frac{z''}{z'}=f(z)$ is 2-step analyzable in the constants. 
\item For generic $c$, $z' = g(z)+c$ is orthogonal to the constants. 
\end{enumerate} 
The goal of this section is to show that all three possibilities can arise in our family of equations.

\subsection{The generic fiber and nonorthogonality to the constants}
The following slightly restated theorem of Rosenlicht gives conditions for a rational order one differential equation to be nonorthogonal to the constants, see \cite{notmin, mcgrail2000search}. 

\begin{thm} \label{rosorth} Let $K$ be a differential field with algebraically closed field of constants. Let $f(z) \in \mc C_K (z) $ and consider the differential equation $z' = f(z)$. Then $z' = f(z)$ is nonorthogonal to the constants if and only if $\frac{1}{f(z)}$ can be written as:
$$c \frac{\pd{u}{z}}{u} \text{ or } c \pd{v}{z} $$ where $c \in \mc C_K$ and $u,v \in \mc C_K(z)$. 
\end{thm}

\begin{lem} \label{notader}
Suppose that $g(z) \in \m C(z)$. Then for $c \in \m C$ generic over the coefficients of $g(z)$, $g(z)+c$ can not be written as $c_1 \pd{v}{y}$ for any $c_1 \in \m C$ and $\pd{v}{y} \in \m C(z)$.
\end{lem}
\begin{proof}
We first establish the following claim:
\begin{claim}
For any $p(z), q(z) \in \m C[z]$ nonzero, sharing no common roots, with at least one of $p,q$ nonconstant, and $c $ generic over the coefficients of $p,q$, the polynomial $p(z) - c q(z)$ has only simple roots. 
\end{claim}
The claim is equivalent to: for some $b \in \m C$, the polynomial $$f(x) = p(z-b) - c q(z-b)$$ has a no constant or linear term (has at least a double root at zero). Then $f(0)=f'(0)= 0.$ It follows that $p(-b)= c q(-b)$ and $p'(-b) = c q'(-b)$. Now, since $p,q$ share no common roots and $c$ is generic over the coefficients of both, we mnust have $q(-b) \neq 0$ and $p(-b) \neq 0.$ Now by a simple computation, it follows that $$\frac{d}{dz} \left( \frac{p(z)}{q(z)} \right) (-b) = 0.$$ Again, since $p, q$ are relatively prime, the function $\frac{p(z)}{q(z)}$ is nonconstant and so $b$ is algebraic over the coefficients of $p,q$. But now $\frac{p(-b) }{q(-b)} = c$, which is impossible as $c$ is generic. This proves the claim. 

From the claim it follows for $g(z) \in \m C(z)$ and $c $ generic over the coefficients of $g$, $\frac{1}{g(z) +c}$ can not be written as $c_1 \pd{v}{y}$ - indeed by the above claim it follows that $\frac{1}{g(z) +c}$ has only simple poles while $c_1 \pd{v}{y}$ has poles of order $2$ or more. 
\end{proof}
Now, combining Lemma \ref{notader} and Theorem \ref{rosorth}, we obtain: 
\begin{cor} \label{logdercon}
If $\frac{z''}{z'}=f(z)$ has rank $2$ and the family of order one subvarieties is given by $z' = g(z)+c$, then the generic solution of $\frac{z''}{z'}=f(z)$ is analyzable in the constants if and only if for generic $c$, $\frac{1}{g(z)+c}$ can be written as $c_1 \frac{\pd{u}{z}}{u} $ for $c_1 \in \mc C_K$ and $u \in \mc C_K(z)$.
\end{cor}

\begin{rem} 
The condition that a rational function can be written as a constant times a single logarithmic derivative is known to be non-constuctible in the coefficients of the rational function - see for instance Corollary 2.10 of \cite{mcgrail2000search}. 
\end{rem} 

\subsection{Internality to the constants}
Consider the case $\frac{z''}{z'}=c$ where $c \in \m C$. In this case, regarding $z$ as a function of $t$ and assuming $c\neq 0$, solutions of the equation can be seen by an elementary calculation, to be given by
$$a e^{ct} + b,$$ for some $a,b \in \m C$. Then, the equation is internal to the constants (with the internality realized over a single solution). This fits into case 1) of the classification given at the beginning of this section. 

\begin{ques}
Is there any any nonconstant rational function $f(z)$ such that $\frac{z''}{z'}=f(z)$ is internal to the constants? 
\end{ques}

\subsection{Analyzability to the constants}
We first fix some notation for this subsection: $\frac{z''}{z'}=f(z)$ with $g(z)$ a rational antiderivative of $f(z)$ so that $z'=g(z)+c$ is a family of order one subvarieties of $\frac{z''}{z'}=f(z)$. 
\begin{lem}
Suppose that $g(z)$ is a degree $2$ polynomial $a_2 z^2 +a_1 z + a_0$. Then for generic\footnote{or more specifically, as long as $c \neq a_0 - \frac{a_1^2 }{4a_2}$} $c \in \m C$, $z'=g(z)+c$ is nonorthogonal to the constants. 
\end{lem}
\begin{proof} 
We have that $$\frac{1}{g(z)+c} = \frac{d}{(z-\alpha) (z-\beta)},$$ with $\alpha \neq \beta $. Then writing $A = \frac{d}{\alpha-\beta }$, $B = \frac{d}{\beta -\alpha}$, $$\frac{1}{g(z)+c} = \frac{A}{z-\alpha} + \frac{B}{z-\beta}.$$ If we take $u(z)=\frac{z-\alpha}{z-\beta}$, then $$\frac{1}{g(z)+c} = A \frac{u'}{u},$$ and so it follows by Theorem \ref{rosorth} that $z'=g(z)+c$
\end{proof} 

We next show that the equation $\frac{z''}{z'} = z$ falls under case 2 of the classification mentioned at the beginning of this section:

\begin{lem}
The generic type of the equation \begin{equation} \label{deg2}
    \frac{z''}{z'}=z
\end{equation}
is 2-step analyzable in the constants and is not internal to the constants.
\end{lem}

\begin{proof}

For a generic solution $z$ of equation \ref{deg2}, set $c(z) = z^2 - 2z' \in \mathbb{C}(z,z')$.  The equation \ref{deg2} implies that 
$$ c(z)' = 0$$
and set 
 $$z_0 = \frac{z - \sqrt{c(z)}}{z + \sqrt{c(z)}} \in \mathbb{C}(z,z')^{alg}$$

A direct computation shows that 
\begin{eqnarray*}
\frac{z_0'}{z_0}  &= \frac{(z - \sqrt{c(z)})'}{z - \sqrt{c(z)}} - \frac{(z + \sqrt{c(z)})'}{z + \sqrt{c(z)}} \\
& = \frac{z'}{z - \sqrt{c(z)}} - \frac{z '}{z + \sqrt{c(z)}} \\
& = \frac{2z'\sqrt{c(z)}}{z^2 - c(z)} = \sqrt{c(z)}
\end{eqnarray*}
where we used the exact formula for $c(z)$ in the computation of the denominator for the last equality.  Since $c(z)$ and hence $\sqrt{c(z)}$ are constants,  it follows that \begin{equation}
    \label{deg22} \left( \frac{z_0'}{z_0}\right)' = 0
\end{equation}

Since for $z$ generic,  $\sqrt{c(z)} \notin \mathbb{C}$,  $z_0$ realizes the generic type of equation \ref{deg22} so that there is an algebraic correspondence between the equations \ref{deg2} and \ref{deg22}.  The equation \ref{deg22} is known to be analyzable in exactly two steps in the constants by \cite{jin2020internality} and therefore so is the equation \ref{deg2}.
\end{proof}
A linear change of variables $z_1=\frac{z-b}{a}$ can be used to give a bijective correspondence between equation \ref{deg2} and any such equation with the right hand side an arbitrary linear function of $z$ over $\m C$:

\begin{cor}
For any $a,b \in \m C$, the generic type of the equation $$
    \frac{z''}{z'}=az+b
$$
is 2-step analyzable in the constants and is not internal to the constants.
\end{cor}

\subsection{Orthogonality to the constants} 
We remind the reader of our general notation: $\frac{z''}{z'}=f(z)$ with $g(z)$ an antiderivative of $f(z)$ so that $z'=g(z)+c$ is a family of order one subvarieties of $\frac{z''}{z'}=f(z)$. In this subsection, we consider the case that $g(z)$ is a degree three polynomial over $\m C$.

\begin{lem}
There is no polynomial $P(z)$ of degree $3$ such that 
$$f_c(z) = \frac 1 {P(z) + c} $$
is a constant multiple of a logarithmic derivative in $\mathbb{C}(z)$ for generic values of $c$. 
\end{lem}

\begin{proof}
By contradiction, assume that such a polynomial $P(z)$ exists.  Without loss of generality, we can assume that $P(z)$ is monic and the constant coefficient of $P(z)$ is $0$.  So we write: 
$$ P(z) = z^3 + az^2 + bz.$$

This implies that the quotients of the residues do not depend on $c$ and therefore that there exists \textit{fixed} $A_1,A_2,A_3 \in \mathbb{C}^\ast$ such that for infinitely many values of $c$, 

$$(\ast): \frac 1 {P(z) + c} = e.(\frac{A_1}{z - \alpha_1} + \frac{A_2}{z - \alpha_2} + \frac{A_3}{z - \alpha_3})$$
for some $e \neq 0, \alpha_1,\ldots, \alpha_3$. So, choose $c$ such that $P(z) + c$ has simple roots (this holds for any $c$ independent from $a,b,$ for instance) and $A_1,A_2,A_3$ are the residues of $f_c(z)$. For $d$ close enough to $c$,  $P(z) + d$ also as simple roots $\beta_1,\ldots, \beta_3$ and if $B_1, B_2,B_3$ are the residues of $f_d(z)$ then 
$$B_2/B_1 = A_2/A_1 \text{ and } B_3/B_1 = A_3/A_1.$$

It follows that 
\begin{eqnarray*}
f_d(z) &= & \frac { B_1}{z - \beta_1} + \frac {B_2}{z - \beta_2} + \frac { B_3}{z - \beta_3} \\
& = & \frac{B_1} {A_1}\Big(\frac { A_1}{z - \beta_1} + \frac {A_2}{z - \beta_2} + \frac {A_3}{z - \beta_3} \Big)
\end{eqnarray*}

Up to replacing $e$ by $e.(A_1A_2A_3)$, we can assume that $A_1,A_2$ and $A_3$ have been chosen such that:
$$ (E_1): A_1A_2A_3 = 1.$$

With this normalization,  we claim that:
\begin{claim}
$A_1,A_2$ and $A_3$ are the three third roots of unity.  In particular, $A_1/A_2 \notin \mathbb{Q}$.
\end{claim}
\begin{proof}
It is enough to show that 
$$\begin{cases} A_1 + A_2 + A_3 = 0 \\
A_1A_2 + A_1A_3 + A_2A_3 = 0.
\end{cases}$$
since this implies that $(z - A_1)(z - A_2)(z - A_3) = z^3 - 1$. Note that  $\alpha_1,\ldots, \alpha_3$ must be the roots of $P(z) + c$ so we get the two equations: 

$$(S):\begin{cases}
\alpha_1 + \alpha_2 + \alpha_3 = a \\
\alpha_1 \alpha_2 + \alpha_2 \alpha_3 + \alpha_1 \alpha_3 = b.
\end{cases}$$

On the other hand, developing $(\ast)$ gives: 

\begin{eqnarray*}
& \frac 1 {(z-\alpha_1)(z - \alpha_2)(z - \alpha_3)}  = e. \frac{A_1(z - \alpha_2)(z - \alpha_3) + A_2(z - \alpha_1)(z - \alpha_3) + A_3(z-\alpha_1)(z - \alpha_2)}{(z-\alpha_1)(z - \alpha_2)(z - \alpha_3)} \\
& =  e.\frac{\Big(A_1 + A_2 + A_3\Big)z^2 - \Big(A_1(\alpha_2 + \alpha_3) + A_2(\alpha_1 + \alpha_3) + A_3(\alpha_1 + \alpha_2)\Big)z + \Big(A_1\alpha_2\alpha_3 + A_2 \alpha_1 \alpha_3 + A_3 \alpha_1 \alpha_2\Big)} {(z-\alpha_1)(z - \alpha_2)(z - \alpha_3)}
\end{eqnarray*}

The coefficients of $z^2$ and $z$ on the right hand side must therefore be $0$ and the constant coefficient must be equal to $1$. The last equation defines $e$ implicitly in terms of the other parameters so we won't be using it.  Next, consider the coefficient of $z^2$
$$(E_2): A_1 + A_2 + A_3 = 0.$$
The sum of the residues is $0$. The coefficient of $z$:
\begin{eqnarray*}
0 = A_1(\alpha_2 + \alpha_3) + A_2(\alpha_1 + \alpha_3) + A_3(\alpha_1 + \alpha_2) & = &  \\ 
\alpha_1 (A_2 + A_3) + \alpha_2 (A_1 + A_3) + \alpha_3(A_1 + A_2) & = &  \\
- \alpha_1 A_1 - \alpha_2 A_2 - \alpha_3 A_3
\end{eqnarray*}
where in the last equality we used $(E_2)$.

Together with the system $(S)$, this yields that $\alpha_1,\alpha_2,\alpha_3$ are solutions of the system of polynomial equations: 

$$(\overline{S}):\begin{cases}
X_1 + X_2 + X_3 = a \\
X_1X_2 + X_2X_3 + X_1X_3 = b \\
 A_1 X_1 +  A_2 X_2 + A_3 X_3 = 0
\end{cases}$$
 This is where we use our assumption: since this is true for infinitely many values of $c$,  this system must have infinitely many solutions so its set of solutions must have dimension $\geq 1$ (actually $= 1$).  This will give us our last equation on $A_1,A_2,A_3$:
 
Consider $q = (q_1,q_2,q_3)$ a common solution of the system above (since it has infinitely many solutions). The first equation and the last equations are equations of planes so they must intersect on a line $L$ of the form 

$$L = \lbrace q + \lambda v,  \lambda \in \mathbb{C} \rbrace$$
where the vector $v$ is given by 
$$ v = \begin{pmatrix} 1 \\ 1 \\ 1 \end{pmatrix} \wedge \begin{pmatrix} A_1 \\ A_2 \\ A_3 \end{pmatrix} = \begin{pmatrix}A_3 - A_2 \\ A_1 - A_3 \\ A_2 - A_1 \end{pmatrix} $$

So in order for the system $(\overline{S})$ to have infinitely solutions this line $(L)$ must be contained in the conic given by the second equation: 

$$ (q_1 + \lambda v_1)(q_2 + \lambda v_2) +  (q_2 + \lambda v_2)(q_3 + \lambda v_3) + (q_1 + \lambda v_1)(q_3 + \lambda v_3) =b$$
So the coefficient in $\lambda^2$ must vanish, which gives: 

\begin{eqnarray*} 
& 0  = v_1v_2 + v_2v_3 + v_1v_3  \\
& = (A_3 - A_2)(A_1 - A_3) + (A_1 - A_3)(A_2 - A_1) + (A_3 -A_2)(A_2 - A_1) \\
& = - (A_1^2 + A_2^2 + A_3^2) + A_1A_3 + A_1A_2 + A_2A_3 \\
& = -(A_1 + A_2 + A_3)^2 + 3(A_1A_3 + A_1A_2 + A_2A_3) \\
& = 3(A_1A_3 + A_1A_2 + A_2A_3)
\end{eqnarray*} 
where we used $(E_2)$ on the last line.  We conclude that 
$$(E_3): A_1A_2 + A_2A_3 + A_1A_3 = 0.$$
\end{proof}

To conclude the proof of the lemma, we use the following argument explained in the Example 2.20 of \cite{HrIt}: every (non constant)  $f(z) \in \mathbb{C}(z)$ can be written as: 
$$f(z) = \frac {(z - a_1) \ldots (z - a_n)} {(z- b_1)\ldots (z-b_m)} $$
By direct calculation, one can see that:
$$\frac{f'(z)}{f(z)} = \sum \frac{1}{z - a_i} - \sum \frac{1}{z - b_i} .$$
It follows that every logarihmic derivative has only simple poles with integer residues.  So if $g(z)$ is a constant multiple of a logarithmic derivative, then all poles are simple are the quotients of the residues are rational, but we've observed that is impossible.
\end{proof}

By combining the previous Lemma with Corollary \ref{logdercon}, we see that
\begin{cor}\label{3const}
For any polynomial $P(z)$ of degree $3$, then for generic $c \in \m C$ independent from the coefficients of $P$, $z'=P(z)+c$ is orthogonal to the constants. 
\end{cor}

\begin{prop} \label{genorth}
Suppose $a,b,c,d$ are algebraically independent over $\Q$.  Let $g(z)=z^3+az^2+bz$.  The  strongly minimal sets defined by $z^\prime=g(z)+c$ and $z^\prime=g(z)+d$ are orthogonal. 
\end{prop}

\begin{proof}
 Let $\alpha_1,\dots,\alpha_3$ be the zeros of $g(z)+c$.  Then $\alpha_1,\dots,\alpha_3$ are algebraically independent.
  $${1\over g(z)+c}=\sum_{i=1}^3 {A_i\over z-\alpha_i}$$ where $$A_i={1\over\prod_{j\ne i}(\alpha_i-\alpha_j)}.$$ 
  
We have the linear relation $A_1+A_2+A_3=0$.   

\begin{claim} \label{cl1} 
If $m_1,m_2,m_3 \in \Q$ and $\sum m_i A_i \in \Q (a,b)^{alg}$, then $m_1=m_2=m_3$.
\end{claim}  
 
Suppose $\sum m_i A_i=\beta\in \Q(a,b)^\alg$.
\begin{eqnarray*}\beta \prod_{j<i}(\alpha_i-\alpha_j) &=& m_1(\alpha_2-\alpha_3)-m_2(\alpha_1-\alpha_3)+m_3(\alpha_1-\alpha_2) \\
&=& (m_3-m_2)\alpha_1+(m_1-m_3)\alpha_2+ (m_2-m_1)\alpha_3 
\end{eqnarray*}

If $m_1=m_2=m_3$, then we are left with the equation $\beta\prod_{j<i}(\alpha_j-\alpha_i)=0$.
Since $\alpha_1,\alpha_2,\alpha_3$ are algebraically independent we must also have $\beta=0$.  Otherwise we have a degree 3 polynomial over $\Q(a,b)^\alg$ vanishing at
$(\alpha_1,\alpha_2,\alpha_3)$.

Let's write the linear term
$\sum n_i\alpha_i$.
We now have the following system of equations over $\Q(a,b)^\alg$ satisfied by $\alpha_1,\alpha_2,\alpha_3$.
\begin{eqnarray*}
-a&=& z_1+z_2+z_3\\
b&=& z_1z_2+z_1z_3+z_2z_3\\
 \beta \prod_{j<i}(z_i-z_j) &=&nz_1+n_2z_2+ n_3 z_3\\
\end{eqnarray*}

Let $H$ be the hyperplane $z_1+z_2+z_3=-a$,  $V$ the surface $z_1z_2+z_1z_3+z_2z_3=b$,
and $W$ the surface  $\prod_{i<j}(z_i-z_j) =nz_1+n_2z_2+ n_3 z_3$.
We will show $H\cap V\cap W$ is finite.  But then $\alpha_1,\alpha_2,\alpha_3\in \Q(a,b)^\alg$,
a contradiction.

We make the substitution $z_3=-z_1-z_2-a$ into the defining equation for $V$ to get
$F(z_1,z_2)=0$ where 
$$F(z_1,z_2)=z_1^2+z_2^2+z_1z_2+az_1+az_2+b.$$
This is an irreducible polynomial.

Making the same substitution into the defining equation for $W$ we get $G(z_1,z_2)=0$
where  
$$G(z_1,z_2)=\beta [2z_1^3-2z_2^2+3z_1^2z_2-3z_1z_2^2]+ \hbox{ lower degree terms}$$

If $F(z_1,z_2)=G(z_1,z_2)=0$ has infinitely many solutions, then, since $F$ is irreducible,
we must have $F|G$.  But comparing the homogeneous parts of $F$ and $G$ of highest degree we see that is impossible, so we've established the claim. 

Now suppose the strongly minimal set $z^\prime=g(z)+c$ and $z^\prime=g(z)+d$ are non-orthogonal.
Let $B_1,\dots,B_n$ be the residues for $g(z)+d$.

By Subsection \ref{5.2} of this paper (or \cite[2.22]{HrIt}), 

$$\ld _\Q (A_1,A_2,A_3,B_1,B_2,B_3) < \ld _\Q (A_1,A_2,A_3) +\ld _\Q (B_1,B_2,B_3)=4.$$ Thus we have an equation $$\sum m_i A_i=\sum n_iB_i$$ where neither $m_1=m_2=m_3$ or 
$n_1=n_2=n_3$. 

Since $A_1,\dots,A_n$ are algebraic over $\Q(a,b,c)$ and  $c\dnf_{\Q(a,b)^\alg}d$,\break
$\tp(A_1,A_2,A_3/\Q(a,b)^\alg,B_1,B_2,B_3)$ is finitely satisfiable  in $\Q(a,b)^\alg$.
Thus we have $\sum m_i A_i\in \Q(a,b)^\alg$ contradicting Claim \ref{cl1}.
\end{proof}

We now derive the following model-theoretic consequence of Corollary \ref{3const} and Proposition \ref{genorth}.  

\begin{cor}
Let $f(z) = z^2 + az + b$ be a complex polynomial of degree $2$. Then the theory of the solution set of
$$(\star): z''/z' = f(z)$$
has the dimensional order property (DOP) and hence $2^\kappa$ isomorphism classes of models of cardinal $\kappa$ for every uncountable cardinal $\kappa$.
\end{cor}

Recall that a complete totally transcendental theory $T$ has the dimensional order property (DOP) if there are models $M_0 \subset M_1,M_2$ with $M_1 \ind_{M_0} M_2$ and a regular type $q$ with parameters in the prime model over $M_1 \cup M_2$ such that $q$ is orthogonal to $M_1$ and $M_2$.  It is well-known that if $T$ has the DOP then $T$ has $2^\kappa$ isomorphism classes of models of cardinal $\kappa$ for every $\kappa \geq \aleph_1 + \mid T \mid$.
\begin{proof}
First note that if $\alpha \neq 0$ then $y \mapsto \alpha y$ gives a definable bijection between the solution sets of $z''/z' = f(z)$ and $z''/z' = f(z/\alpha)$.  Choosing $\alpha = 1/\sqrt{3}$, we can assume that $f(z)$ is of the form
$$f(z) = 3z^2 + \sqrt{3}az + b$$
Set $g(z) = z^3 + \frac{a\sqrt{3}}{2}z^2+ bz$ and let $c$ be a transcendental constant over $\mathbb{Q}(a,b)$.  We claim that  the generic type $q_c \in S(\mathbb{Q}(a,b,c))$ of
$$ z' = g(z) + c.$$
is orthogonal to $\mathbb{Q}(a,b)^{alg}$:  assume that $q_c$ is non-orthogonal to $\mathbb{Q}(a,b)^{alg}$.  Since $q_c$ is strongly minimal and orthogonal to the constants by Corollary \ref{3const}, $q_c$ is one based. It follows that there exists a \emph{minimal} type $q_0 \in S(\mathbb{Q}(a,b)^{alg})$ non-orthogonal to $q_c$.  Moreover,  any copy $q_{d}$ of $q_c$ for every transcendental constant $d$ over $\mathbb{Q}(a,b)$ is also non-orthogonal to $q_0$.  

By transitivity of the non-orthogonality relation for minimal types,  the types $q_c$ and $q_{d}$ are non-orthogonal whenever $c$ and $d$ are transcendental constant over $\mathbb{Q}(a,b)$.  This contradicts Proposition \ref{genorth}, hence $p$ is orthogonal to $\mathbb{Q}(a,b)^{alg}$.

We conclude as in Chapter 3, Corollary 2.6 of \cite{MMP} that the theory of the solution set of $(\star)$ has the DOP: consider $M_0$ the prime model over $\mathbb{Q}(a,b)$,  $c$ and $d$ independent transcendental constants over $\mathbb{Q}(a,b)$ and denote by $M_1$ (resp.  $M_2$) the prime model over $\mathbb{Q}(a,b,c)$ (resp. $\mathbb{Q}(a,b,d)$).  

Set $e = c + d$ and $q = q_e$. We claim that $q_e$ is orthogonal to both $M_1$ and $M_2$: since $e$ is a transcendental constant over $M_1$, we have that: 
$$e \ind_{\mathbb{Q}(a,b)^{alg}} M_1 \text{ and } q_e \text{ orthogonal to } \mathbb{Q}(a,b)^{alg}$$
from which it follows that $q_e$ is orthogonal to $M_1$. Similarly,  $q_e$ is orthogonal to $M_2$, hence the theory of the solution set of $(\star)$ has the DOP and hence the maximal number of isomorphism classes of models in any uncountable cardinal.  
\end{proof}
In particular,  from our analysis of a specific autonomous second order equation,  we recover Shelah's theorem in \cite{shelah1973differentially} which asserts that the theory $\textbf{DCF}_0$ admits the maximal number  of isomorphism classes of models in any given uncountable cardinal.  While Shelah's proof uses differentially transcendental elements,  it was already noticed by Poizat in \cite[pp. 10]{poizat1980c} that the DOP is also witnessed by families of algebraic differential equations parametrized by constants such as: 

$$ (x' = {cx\over 1+x}, c \in {\mathcal C}^\times).$$

\begin{rem} In the same vein, it is interesting to note that our results allows us to compute effectively the oldest model-theoretic invariant --- the function $\kappa \mapsto I(\kappa)$ which counts the isomorphism classes of models of cardinal $\kappa$ --- for the solution sets of equations of the form $(\star)$. More precisely, if  $T_f$  denotes the theory of the solution set of $y''/y' = f(y)$ and $I(\kappa,T_f)$ counts the number of isomorphism classes of models of $T_f$ of cardinal $\kappa$ then:
\begin{itemize}
\item[(1)] The rational function $f(z)$ is a derivative in $\mathbb{C}(z)$ if and only if $I(\kappa,T_f) = 1$ for all infinite cardinals $\kappa$.

\item[(2)] If $f(z)$ is constant or a linear polynomial then  $I(\kappa,T_f) = 1$ for all uncountable cardinals $\kappa$ but $I(\aleph_0,T_f) = \aleph_0$.

\item[(3)] If $f(z)$ is a polynomial of degree $2$ then $I(\kappa,T_f) = 2^\kappa$ for every uncountable cardinal $\kappa$.
\end{itemize}
From this perspective,  it would be interesting to show that no other function $\kappa \mapsto I(\kappa)$ can occur in the family $(\star)$ or equivalently that the theory of the solution set of any differential equation of the form $(\star)$ which is not analyzable in the constants nor strongly minimal admits the maximal number of isomorphism classes of models in every uncountable cardinal. 
\end{rem}

\bibliography{Research}{}
\bibliographystyle{plain}
\end{document}